\documentclass[11pt]{article}

\usepackage{hyperref}
\usepackage{amsmath, amssymb, amsthm}
\usepackage{cmap}
\usepackage{color}
\usepackage{graphicx}
\usepackage[dvipdfm,a4paper,left=30mm,right=30mm,top=35mm,bottom=35mm]{geometry}
\usepackage[dvipsnames]{xcolor}
\usepackage[utf8]{inputenc}

\usepackage{IEEEtrantools}
\usepackage{cases}
\usepackage{comment}
\usepackage{bm}

\newtheorem{theorem}{Theorem}
\newtheorem{proposition}{Proposition}
\newtheorem{example}{Example}
\renewenvironment{proof}{ \noindent {\bfseries Proof.}}{qed}
\newtheorem{remark}{Remark}
\newtheorem{lemma}{Lemma}
\newtheorem{corollary}{Corollary}

\def\Div{\text{div}}

\def\sign{{\rm sign}}

\def\eps{\varepsilon}
\def\C{\hbox{\rlap{\kern.24em\raise.1ex\hbox
      {\vrule height1.3ex width.9pt}}C}}
\def\R{\mathbb{R}}
\def\S{\mathbb{S}}
\def\I{{\rm I\kern-.2em I}} 
\def\E{{\rm I\kern-.2em E}} 
\def\P{\hbox{\rlap{I}\kern.16em P}}
\def\Q{\hbox{\rlap{\kern.24em\raise.1ex\hbox
      {\vrule height1.3ex width.9pt}}Q}}
\def\M{\hbox{\rlap{I}\kern.16em\rlap{I}M}}
\def\N{\hbox{\rlap{I}\kern.16em\rlap{I}N}}
\def\Z{\hbox{\rlap{Z}\kern.20em Z}}

\def\({\begin{eqnarray}}
\def\){\end{eqnarray}}
\def\[{\begin{eqnarray*}}
\def\]{\end{eqnarray*}}
\def\part#1#2{{\partial #1\over\partial #2}}

\def\dt{\partial_t}
\def\pd{\partial}

\def\H{\mathcal{H}}
\def\pmb#1{\setbox0=\hbox{$#1$}
  \kern-.025em\copy0\kern-\wd0
  \kern-.05em\copy0\kern-\wd0
  \kern-.025em\raise.0433em\box0 }
\def\bar{\overline}

\def\p{\varphi}
\def\R{\mathbb{R}}

\newcommand{\dd}{\, \text{d}}
\newcommand{\ve}{\tilde v}

\title{Large-scale dynamics of self-propelled particles moving through obstacles: model derivation and pattern formation} 
\author{ P. Aceves-Sanchez$^1$, P. Degond$^2$, E. E. Keaveny$^2$,\\ A. Manhart*$^{3}$, S. Merino-Aceituno$^4$, D. Peurichard$^5$} 
\date{} 

\begin{document}

\maketitle

\begin{center}
1. Department of Mathematics, North Carolina State University,\\
Raleigh, North Carolina 27695, USA,\\ paceves@ncsu.edu
\end{center}

\begin{center}
2. Department of Mathematics, Imperial College London\\
London, SW7 2AZ, United Kingdom,\\ p.degond@imperial.ac.uk, e.keaveny@imperial.ac.uk\\
\end{center}

\begin{center}
3. Department of Mathematics, University College London, \\
London, WC1H 0AY, United Kingdom,\\ * \textit{corresponding author}, a.manhart@ucl.ac.uk\\
\end{center}

\begin{center}
4. Faculty for Mathematics, University of Vienna, \\
Oskar-Morgenstern-Platz 1, Vienna, 1090, Austria\\
Department of Mathematics, University of Sussex\\
Brighton, BN1 9RH, United Kingdom, \\sara.merino@univie.ac.at
\end{center}

\begin{center}
5. Inria Paris - MAMBA project team, Sorbonne University,\\ Laboratory Jacques-Louis Lions, Paris, France,\\ diane.a.peurichard@inria.fr
\end{center}

\noindent{\bf Abstract.} We model and study the patterns created through the interaction of collectively moving self-propelled particles (SPPs) and elastically tethered obstacles. Simulations of an individual-based model reveal at least three distinct large-scale patterns: travelling bands, trails and moving clusters. This motivates the derivation of a macroscopic partial differential equations model for the interactions between the self-propelled particles and the obstacles, for which we assume large tether stiffness. The result is a coupled system of non-linear, non-local partial differential equations. Linear stability analysis shows that patterning is expected if the interactions are strong enough and allows for the predictions of pattern size from model parameters. The macroscopic equations reveal that the obstacle interactions induce short-ranged SPP aggregation, irrespective of whether obstacles and SPPs are attractive or repulsive.
\vskip 1cm
\noindent{\bf Key words:} Self-propelled particles; hydrodynamic limit; pattern formation; stability analysis; gradient flow; non-local interactions
\medskip

\noindent{\bf AMS subject classification:} 35Q70; 82C05; 82C22; 82C70; 92B25; 92C35; 76S05;         
\medskip

\noindent{\bf Acknowledgements:}
SMA is supported by the Vienna Science and Technology Fund (WWTF) with a Vienna Research Groups for Young Investigators, grant VRG17-014. PD acknowledges support by the Royal Society and the Wolfson Foundation through a Royal Society Wolfson Research
Merit Award no. WM130048 and by the National Science Foundation (NSF) under
grant no. RNMS11-07444 (KI-Net). PD is on leave from CNRS, Institut de
Math\'ematiques de Toulouse, France. PD, EEK, PAS, SMA and AM acknowledge support from Engineering and Physical Sciences Research Council (EPSRC) grant EP/P013651/1. PAS, SMA and AM acknowledge that part of the work was done at Imperial College London. PD, PAS and SMA also acknowledge support from EPSRC grant EP/M006883/1. 

\medskip
\noindent{\bf Data statement:} No new data were collected in the course of this research.


\section{Introduction}

This work is devoted to deriving and analysing a model of collectively moving self-propelled particles that interact with a complex, heterogeneous environment. The field of collective dynamics studies what happens when a large number of agents, which can be animals, people, micro-organisms, crystals, etc., interact with each other. A particular focus is the emergence of large scale order or patterns. Famous examples include global alignment in crystals \cite{DeGennes_Prost_1993}, lane formation for people \cite{Feliciani2016}, waves and aggregation in bacteria \cite{Shimkets1990, Ben-Jacob2000}, milling in schools of fish \cite{Shaw1978} or swarming in birds \cite{Cavagna2010}. All these examples have in common that local, small-scale interaction rules between individuals lead to global, large-scale patterns. These patterns are typically hard or impossible to predict from the local interaction rules, hence their understanding requires the use of either extensive simulations or mathematical analysis.

\paragraph{Combining collective dynamics and environmental effects.} In many systems one also needs to take into account the environment to be able to explain observed patterns in collective phenomena \cite{chepizhko2013optimal,chepizhko2013diffusion,
jabbarzadeh2014swimming,park2008enhanced}. For cells moving through a tissue, this environment often includes fibres and other components. For instance, it has been observed that many cell types have a tendency to move up stiffness gradients, a phenomenon termed \textit{durotaxis} \cite{Lo2000}. In some of these instances the effect on the substrate is negligible. However in many applications the interaction modifies the environment (either permanently or transiently) in a way that affects subsequent interactions. An example is the degradation of the extracellular matrix (ECM) caused by migrating cells \cite{Baricos1995}, which affects the ECM structure and hence future migration. In this work we want to combine collectivity and environmental interactions and study the resulting patterns. Known examples of patterns created include travelling bands of large swarms of scavenging locust \cite{Buhl2006, Topaz2008}, the formation of paths in grass-land by active walkers \cite{Helbing1997,Lam1995} or aggregation of individuals \cite{Bernoff2016}. For metastasising cancer cells it was observed that the invasion success depends on whether they move individually or as small clusters \cite{Cheung2016}.

\paragraph{Obstacles can emulate complex environments.} 
The importance of the environment is particularly true for sperm dynamics, where the surrounding fluid plays a key role in the emergence of collective motion. For example clustering and large-scale swirling was observed in simulations of collectively moving sperm in \cite{Schoeller2018,Sokolov2007}. In \cite{Degond2019} a model was proposed that couples the Vicsek model for collective dynamics with Stokes equations for a viscous fluid. However, sperm dynamics takes place in a complex fluid, whose constitutive properties cannot be characterised solely by a viscosity. To approximate the complex environment the introduction of immersed obstacles has been proposed \cite{kamal2018enhanced,majmudar2012experiments,wrobel2016enhanced}. For example, in \cite{kamal2018enhanced} the authors propose a model in which an undulatory swimmer swims in a fluid filled with elastically tethered obstacles, however effects of collective dynamics, i.e. multiple swimmers, were not investigated. In this article we present a model for collective motion in an environment filled with spheres tethered to fixed points in space via linear springs, that play the role of obstacles. We will study the impact of this obstacle-based environment on the collective dynamics for a  large number of self-propelled particles (SPPs).

\paragraph{Individual vs continuum description.} From a modelling perspective, two approaches are common \cite{Mogilner2016}. One can formulate a system of individual-based models (IBMs), also called agent-based models, where the behaviour of each individual is assumed to be governed by separate, often stochastic ordinary differential equations. This approach has the advantage that the translation of modelling assumptions of the individual level is relatively straight-forward. However, few analytical tools are available to study IBMs and even if the system exhibits the desired property, limited insight can be gained as to why it does so. On the other hand, one can formulate a partial differential equation (PDE) model for the macroscopic quantities of interest, e.g. the space and time dependent density of agents. A rich mathematical toolbox exists for the analysis of PDEs, which includes linear stability analysis, constructions of steady states as well as efficient simulation tools. Substantial progress has been made to establish systematic links between IBMs and the corresponding PDEs, \cite{Degond2008, Ha2008}. This allows to combine the advantages of both methods: straight forward translation of biological assumptions into the IBM, and strong analytical tools for the PDE model. The self-organised hydrodynamics (SOH) approach \cite{Degond2008} used in this work has been successfully applied e.g. to fibre interactions \cite{Peurichard2016}, bacterial swarms \cite{Manhart2018}, sperm fertility \cite{Creppy2016} or ant trail formation \cite{Boissard2013}.

\paragraph{Paper structure.} In Sec.~\ref{sec:IBM} we present the individual-based model, at whose basis lies the famous Vicsek model \cite{Vicsek1995}. This model describes SPPs that align their orientation with neighbouring particles, to which we add a short ranged repulsion term. The environment consists of obstacles which are tethered via linear springs to anchor points fixed in space. SPPs and obstacles exert either repulsive or attractive forces on each other. Simulations of the IBM reveal the richness of possible patterns for this simple system, which includes clustering, trail formation and travelling bands, and motivate the formulation of a macroscopic PDE model of the SPP-obstacle interactions. The derivation of the macroscopic model, presented in Sec.~\ref{sec:derivation}, builds on the SOH technique for the SPPs, but requires new techniques for the obstacles. We focus on a particular asymptotic regime, where the obstacle tethering is strong, i.e. strong spring stiffness. The derived macroscopic model for SPP-obstacle interactions is presented and interpreted in Sec.~\ref{ssec:macroswimmer} and the main theorem is proven in Sec.~\ref{ssec:proof}. We capitalize on the macroscopic model by analysing pattern formation through linear stability analysis in Sec.~\ref{ssec:stabAnalysis}. In Sec.~\ref{ssec:micromarcokernel} we use the macroscopic model to discover that obstacles mediate an effective SPP interaction with biphasic behaviour. Finally in Sec.~\ref{sec:numerics} we perform simulations in one space dimension of the macroscopic and individual-based model and compare the results to each other and the analytical results.

\section{The Individual-Based Model (IBM)}
\label{sec:IBM}

\begin{figure}[t]
\centering
\includegraphics[width=0.9\textwidth]{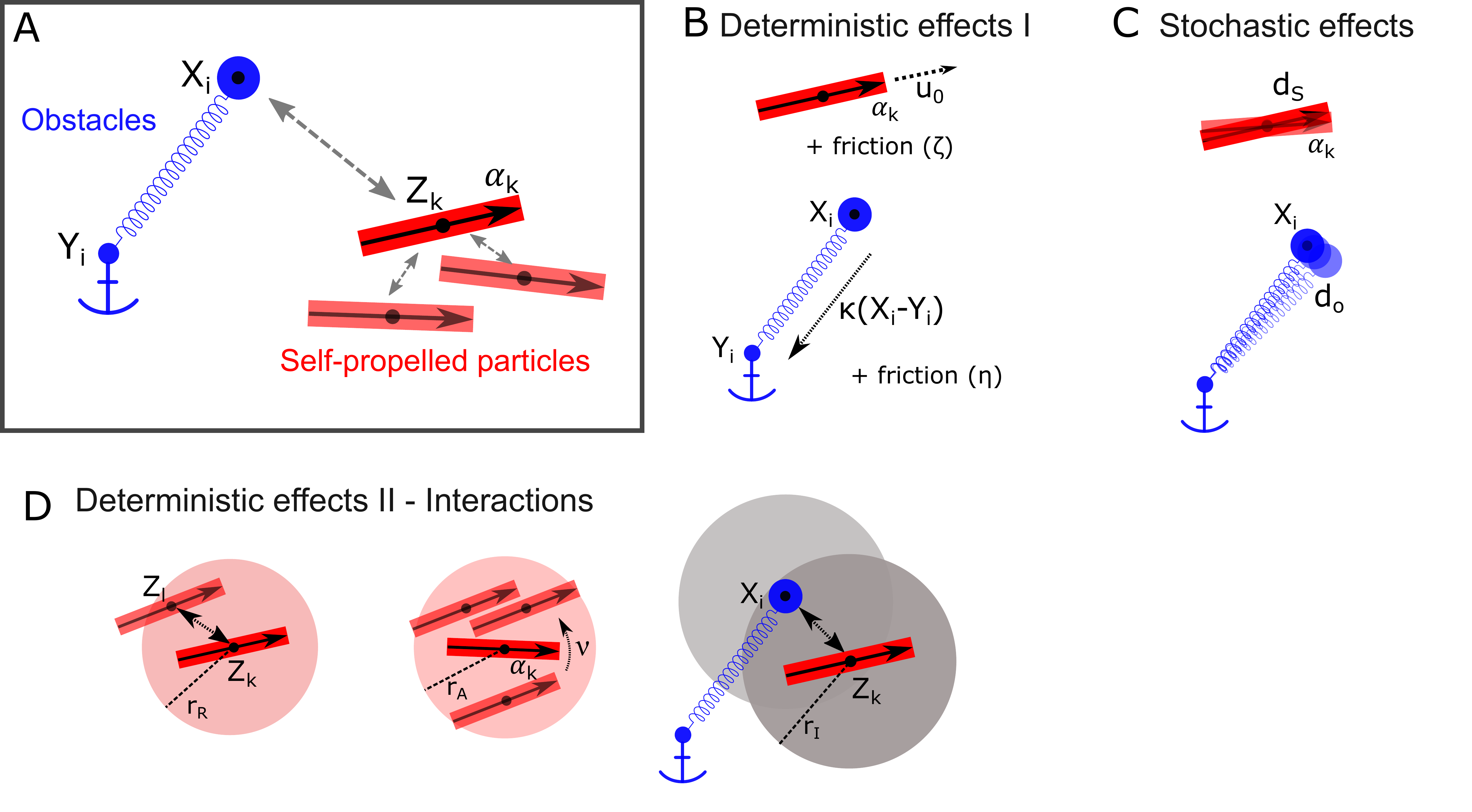}
\caption{\small{Ingredients of the IBM. A: Show are the two types of agents of the IBM, the SPPs (red) and the obstacles (blue). B: Deterministic effects that affect each agent individually. SPPs self-propel themselves and experience friction, obstacles are elastically tethered to their anchor points and also experience friction. C: Stochastic effects for the SPPs (orientation) and the obstacles (position). C: Interactions include SPP repulsion, SPP alignment and SPP-obstacle interactions.}}
\label{fig:IBM}
\end{figure}

\subsection{Formulation of the IBM}

The starting point for our investigation is an individual-based model (IBM), in which the dynamics of each component is described by individual equations coupled through interaction terms. We couple the famous Vicsek model for collective movement of self-propelled particles (SPPs) \cite{Vicsek1995} with an environmental model, described by elastically tethered obstacles. Our IBM is set in $n$-dimensional space, where $n=1,2$ or $3$. The two components and interactions are depicted schematically in Fig.~\ref{fig:IBM}. Several applications of collective movement, in particular when applied to cells, take place at the micro-scale. These regimes are typically friction dominated with negligible inertia (also called over-damped regime). We therefore formulate our model in this friction dominated regime.

\paragraph{Model Components.} We model the following two types of agents:
\begin{itemize}
\item {\it Obstacles:} We consider a set of $N$ mobile obstacles with positions $X_i(t)\in \R^n$ for $i = 1, 2, \ldots, N$ and time $t\geq 0$. Each obstacle is tethered to a fixed anchor point $Y_i\in \R^n$ through a Hookean spring with stiffness constant $\kappa>0$ and experiences friction with the environment with friction constant $\eta>0$.
\item {\it SPPs:} We denote by $Z_k(t)\in \R^n$ the positions of the $k$-the SPP at time $t \geq 0$ for $k = 1, 2, \ldots , M$. Each SPP has a body orientation $\alpha_k(t)\in \mathbb{S}^{n-1}$ and a self-propulsion speed $u_0$ in direction $\alpha_k$. SPPs experience friction with the environment with friction constant $\zeta>0$.
\end{itemize}

\paragraph{Interactions.} We consider the following interactions:
\begin{itemize}
\item {\it SPP alignment:} We assume each SPP aligns its body orientation $\alpha_k$ to the mean orientation $\bar\alpha_k$ of body directions of SPPs in its neighbourhood with radius $r_A$. This happens with an alignment frequency $\nu>0$ and is analogous to the famous Vicsek model for collective swarming \cite{Vicsek1995}.
\item {\it SPP repulsion:} SPPs repel each other at short distances, which models size-exclusion effects. Following \cite{Degond2015} we model this by an even pushing potential $\psi: \R^n \mapsto \R$ with typical spatial scale $r_R>0$. The force felt between two SPPs positioned at $Z_i$ and $Z_j$ is then given by $\nabla \psi \left(Z_i-Z_j\right)$.
\item {\it Obstacle-SPP interaction:} We assume the obstacles and SPPs exert a force on each other, which depends on the the distance between them. Similar to the SPP repulsion, we describe this by an even interaction potential $\phi: \R^n \mapsto \R$ with typical scale $r_I$, yielding the force $\nabla \phi \left(Z-X\right)$ for a SPP at position $Z$ and an obstacle at position $X$. In general we assume this force to be repulsive, however we will discuss the effect of an attractive force in Sec.~\ref{sec:1dAnalysis}.
\end{itemize}

\paragraph{Stochasticity.} We include two sources of uncertainty, both modelled by independent Brownian motions: Stochastic effects in the obstacle position (with intensity $d_o$) as well stochastic effects in the SPP orientation (intensity $d_s$).

\paragraph{Model Equations.} The effects described above can be modelled through the following coupled, stochastic ODEs. Note that in the absence of obstacles, the equations reduce to the time-continuous Viscek model, described e.g. in \cite{Vicsek1995}. From here on we work with the non-dimensional variables (but keeping the same names as introduced above), in particular we haven chosen the domain size $L$ as reference length and $L/u_0$ as reference time. The latter can be interpreted as the time it takes a freely moving SPP to cross the domain. We then obtain:

\begin{subequations}
\label{eqn:IBMmodel_massless}
\begin{align}
 \dd X_i                                 =& -\frac{\kappa}{\eta}(X_i-Y_i)\dd t -\frac{1}{\eta}\frac{1}{M}\sum_{ k = 1}^M \nabla \phi \left(X_i - Z_k\right) \dd t+ \sqrt{ 2 d_o} \, \dd B^{i}_t ,\\
  \dd Z_k                                 =& \alpha_k \dd t - \frac{1}{\zeta}\frac{1}{N}\sum_{i = 1}^N \nabla \phi \left(Z_k - X_i\right)\dd t -\frac{1}{\zeta}\frac{1}{M} \sum_{l \neq k}^M \nabla \psi \left(Z_k - Z_l\right)\dd t,\\
  \dd \alpha_k                         =& P_{ \alpha_k^\perp} \circ \bigg[ \nu \bar{\alpha}_k \dd t + \sqrt{2 d_s} \, \dd \tilde B^{k}_t \bigg]\label{eqn:IBMorient},
\end{align}
\end{subequations}
where the mean direction $\bar{\alpha}_k$ is defined via the mean flux $J_k$ by
\begin{align}
\label{eqn:meanDir}
\bar{ \alpha}_k = \frac{ J_k}{ | J_k|},\quad \text{ where } J_k =\mkern-18mu \sum_{\substack{j = 1 \\ |Z_k-Z_j|\leq r_A}}^M \mkern-18mu \alpha_j.
\end{align}
The tether positions $Y_i$ are given and do not change in time. The operator $P_{\alpha_k^\perp}$ in \eqref{eqn:IBMorient} in an orthogonal projection onto $\alpha_k^\perp$ and ensures that if $\alpha_k(0)\in \mathbb{S}^{n-1}$, then $\alpha_k(t)\in \mathbb{S}^{n-1}$ for all time. Note that we have scaled the interaction terms by the number of SPPs or obstacles to prepare for the kinetic limit of Sec.~\ref{ssec:meanField}.

\begin{remark}[Modelling choices]
In an attempt to create a minimal model, we did not include a number of effects. For example, as opposed to \cite{Degond2015}, we don't model relaxation of the SPP orientation to the SPP velocity. Notice also that we did include repulsion between SPPs, but not repulsion between the obstacles. The former helps avoid collapse of the SPP density. For the obstacles on the other hand, this seems to be less likely due their tethering in space. Also, there is no coupling to a surrounding fluid, which will be subject of future work.
\end{remark}

\subsection{Simulations of the 2D IBM}
\label{ssec:IBMsim}

We simulate the IBM \eqref{eqn:IBMmodel_massless} in two space dimensions. In this work, instead of doing a more thorough investigation, we want to showcase what types of patterns can be created based on the environmental interactions, emphasising the need for a PDE-based description.

\paragraph{Simulation set-up.} We simulate the IBM using $N=5000$ SPPs and $M=5000$ obstacles on a 2D square domain $\mathcal{B}=[0,1]\times[0,1]$ with periodic boundary conditions. We distribute the fixed anchor points $Y_i$ using a uniform distribution on $\mathcal{B}$ and initialize the obstacle positions with $X_i(0)=Y_i$. Initial SPP positions $Z_k(0)$ and orientations $\alpha_k(0)=(\cos(\p_k), \sin(\p_k))$ are both chosen at random with uniform distributions on $\mathcal{B}$ for $Z_k(0)$ and on $[0,2\pi]$ for $\p_k$. For the interaction potentials we use kernels of the following shape
\begin{align*}
&\phi(x)=\frac{3A_I}{2 r_I^3 \pi}(r_I-|x|)^2 H(r_I-|x|),\quad \psi(x)=\frac{3A_R}{2 r_R^3 \pi}(r_R-|x|)^2 H(r_R-|x|),
\end{align*}
where $H(x)$ is the Heaviside function. These kernels are compactly supported on balls with radius $r_I$ and $r_R$ respectively and chosen to yield a continuous pushing force decreasing linearly. They are normalized such that the force mass is $A_I$ and $A_R$ respectively. For simplicity we choose all interaction radii to be the same, i.e. $r_I=r_A=r_R$. We leave the following parameters constant: $d_o=0$, $\eta=1$, $d_s=0.1$,  $A_I=1$.  We're left with five parameters: $\kappa$, $\zeta$, $\nu$, $r_I$ and $A_R$.

\begin{figure}[t]
\centering
\includegraphics[width=0.95\textwidth]{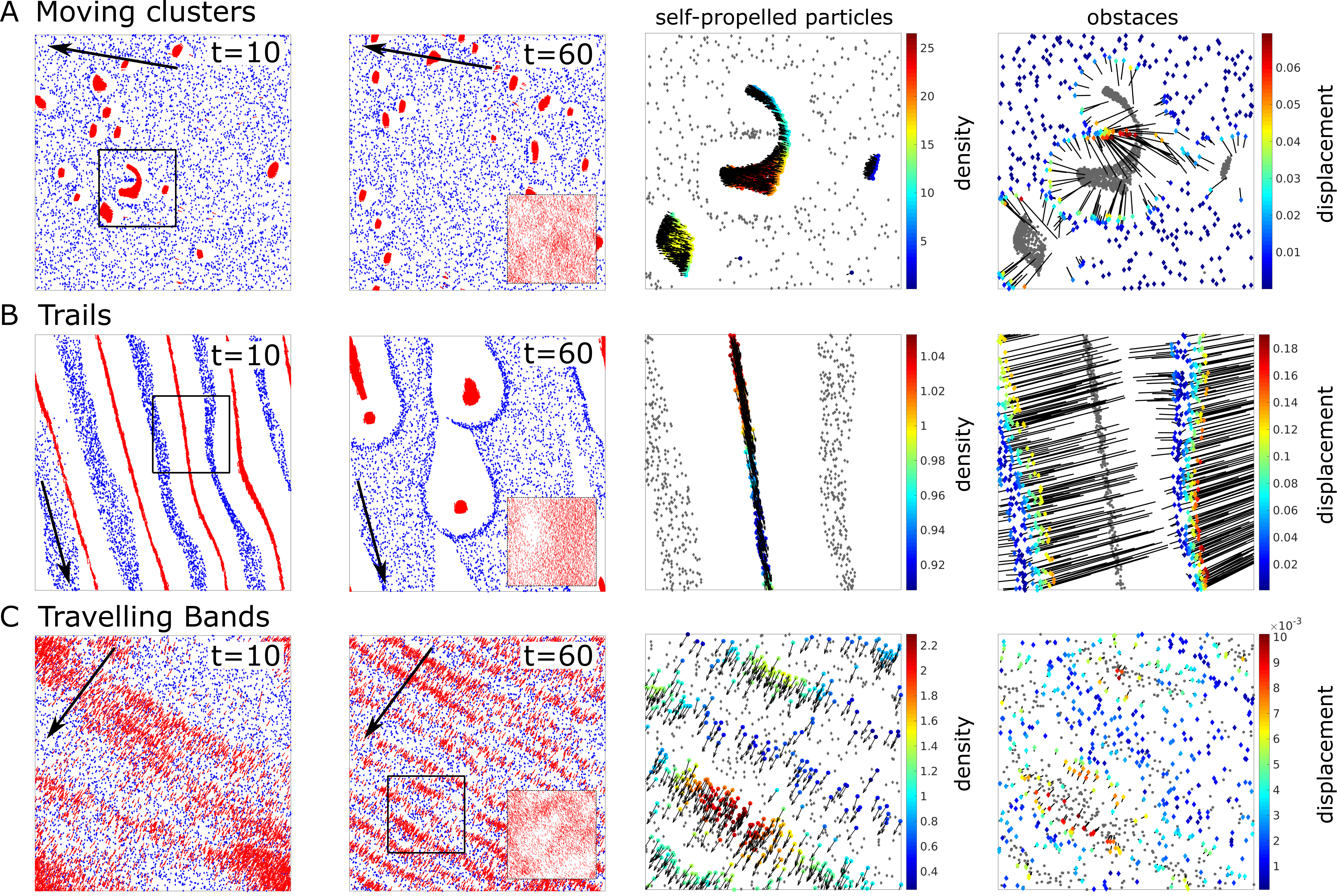}
\caption{\small{IBM patterns. Depicted are snapshots for three different example patterns from the IBM simulations: Moving clusters (A), trails (B) and travelling bands (C). The first two columns show the  SPPs (red arrows) and the obstacles (blue diamonds) in the full 2D simulation domain at two different time points, black arrows mark the mean SPP direction. Insets in the second column show resulting SPP positions for simulations without obstacles. The last two columns show enlargements of the black box in the first columns. In the `SPP' column, the obstacles are shown in grey and the SPPs as black arrows, colors mark SPP neighbourhood density. In the `obstacles' column, SPPs are shown in grey. The lines connect each obstacle to their anchor point, color marks obstacle displacement. Videos can be found in the Supp.~Mat.}}
\label{fig:IBM_patterns}
\end{figure}

\paragraph{IBM simulation results.} Figs.~\ref{fig:IBM_patterns} shows examples of the different patterns produced by different choices of parameters and Fig.~\ref{fig:IBM_stats} shows some associated statistics. Corresponding videos can be found in the Supp.~Mat. The second row in Fig.~\ref{fig:IBM_stats} shows that in all three cases SPPs globally align, i.e. the variance in SPP direction decreases. We call the observed patterns: {\it Moving clusters} ($\kappa=100$, $\zeta=1$, $\nu=10$, $r_I=0.05$, $A_R=0.01$), {\it Trails}  ($\kappa=2.5$, $\zeta=10$, $\nu=100$, $r_I=0.15$, $A_R=0.002$) and {\it Travelling bands} ($\kappa=100$, $\zeta=40$, $\nu=10$, $r_I=0.05$, $A_R=0.002$) and give a short description of them.\newline

\textit{Moving clusters:} In this regime, tether stiffness and SPP-obstacle repulsion is high. The SPPs form very high density groups moving through the obstacles, whose displacement from the anchor points is relatively low. In Fig.~\ref{fig:IBM_patterns}A, we see how a larger cluster is split into two due to the obstacles, suggesting that the cluster size is controlled by the dynamics. This might also be the reason for the relatively large changes in mean SPP density over time seen in Fig.~\ref{fig:IBM_stats}. Nevertheless this pattern seems to be stable.\newline

\textit{Trails:} Here, SPP alignment is strong, with low tether stiffness. The SPPs form stripes parallel to their movement direction, which at $t=10$ seem to very regularly spaced. Within the stripes the SPPs are close together and consequently push the obstacles away from the trails, leading to large obstacle displacements. Interestingly, the trails become unstable and by $t=60$, the SPPs form moving groups. We see this instability building up and the trails falling apart around $t=26$ in Fig.~\ref{fig:IBM_stats}B. The enlargements in Fig.~\ref{fig:IBM_patterns} indicate that the instability of the trails might stem from the fact that the obstacles are not symmetrically displaced to the right and left of the moving trails.\newline

\textit{Travelling bands:} In this pattern the spring strength is high and obstacle displacements are consequently small. SPPs now form bands normal to their direction of movement. At $t=60$ we see in Fig.~\ref{fig:IBM_patterns} that there appears to be a typical spacing between the bands. These patterns seem to be stable.\newline

\begin{figure}[t]
\centering
\includegraphics[width=0.8\textwidth]{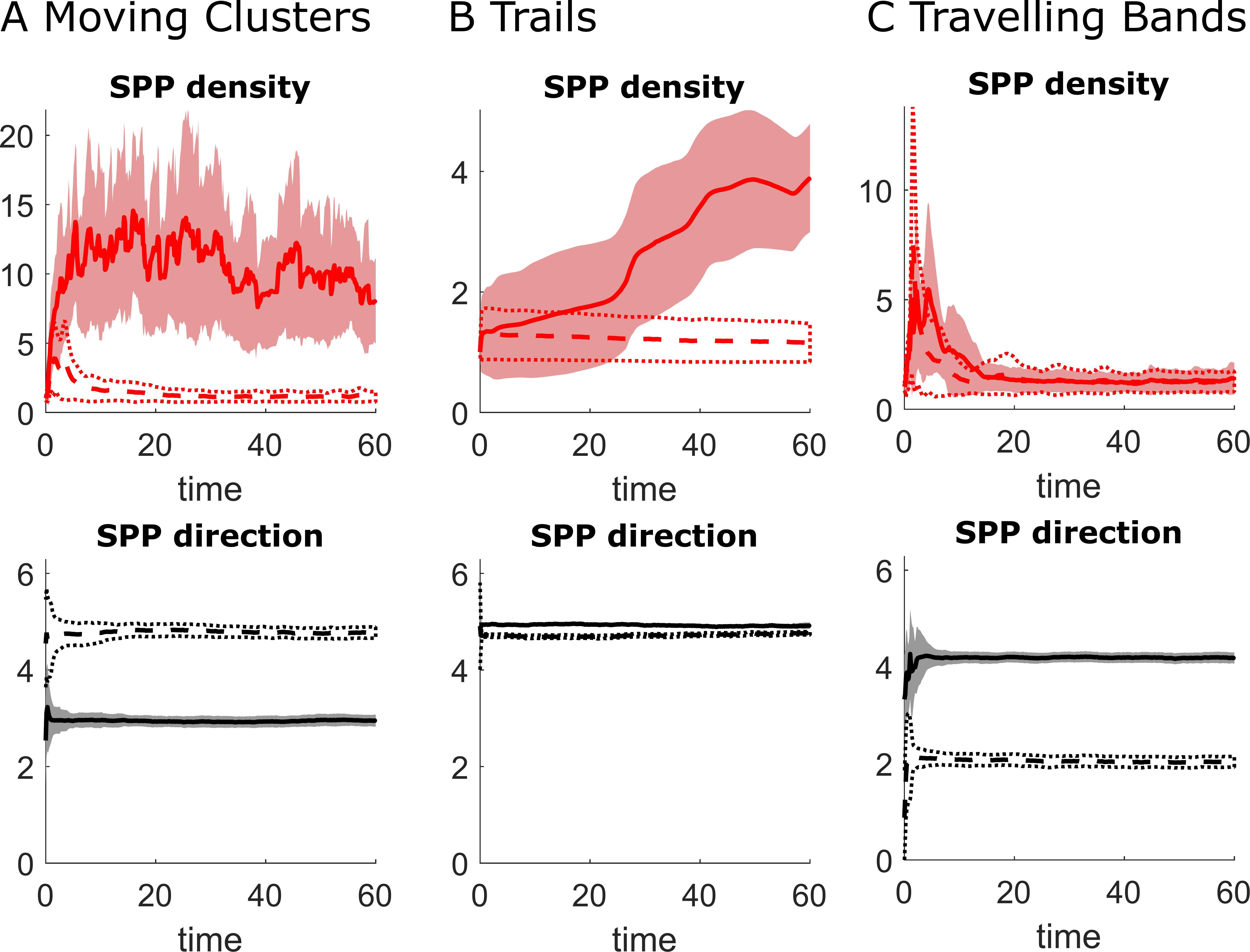}
\caption{\small{IBM statistics. Shown are some statistics for the IBM patterns in Fig.~\ref{fig:IBM_patterns} for moving clusters (A), trails (B) and travelling bands (C). Solid and dashed lines mark averages for simulations with and without obstacles respectively, shaded areas and dotted lines corresponding averages$\pm$standard deviations. SPP densities are calculated for each SPP by calculating the density within a disc of radius $r_A=r_R$ and dividing by the mean density in the domain.}}
\label{fig:IBM_stats}
\end{figure}

\paragraph{Obstacles reinforce and diversify patterns.} To assess the influence of the environment on the pattern formation we compare to simulations of the model without obstacles, i.e. pure Vicsek type dynamics with small SPP repulsion. In the inset in the second column in Fig.~\ref{fig:IBM_patterns} we see that in all three examples there is no patterning in absence of the obstacles. Fig.~\ref{fig:IBM_stats} shows that the alignment behaviour seems unaltered by the obstacles, however for moving clusters and trails the obstacles lead to much higher SPP densities. In is known that for some specific ranges of parameters, clusters and bands already appear in simulations of the Vicsek model alone \cite{Vicsek1995}. However, in the presence of obstacles their qualitative behaviour is different: the environment seems to reinforce such structures and the travelling bands appear to be regularly spaced, which is not the case for bands in the Vicsek model alone. In addition the homogeneous phase (common to the Vicsek dynamics) appears to be less common here. Finally, we observe that also a completely new pattern emerges: trails. 

\paragraph{The need for a PDE description.} The three patterns founds by simulating the IBM show that the interactions between SPPs and obstacles can lead to a rich repertoire of patterns such as clustering, trails and travelling bands. While the system is relatively simple, the number of parameters make is prohibitively expensive to explore fully the complete parameter space. These patterns were found by rough and preliminary parameter scans and we expect that there exist in fact many more patterns. For each example pattern a number of questions arise: Clusters:  It seems that large clusters are split and that there is an intrinsic cluster size. If that is the case how is cluster size controlled and how is it determined from parameters? Trails: The observed trails appear to be a transient, unstable pattern. What makes them unstable and can other parameters produce stable trails? Travelling bands: How is this pattern created and what determines the wavelength and stability?\newline

All these questions suggest that a continuous, PDE-based description of the system is crucial to understanding the observed patterns, as well as to discover others. A PDE-description has several advantages: Patterns such as travelling bands can be constructed explicitly and a stability analysis can performed. Further the PDE description is inherently an averaging process reducing the number of parameters. Lastly, since instead of numerically solving thousands of coupled ODEs, one has to solve only a few PDEs, which makes the simulations much more efficient. The next section is therefore devoted to the derivation of the PDE-based description of the SPP-obstacle model.

\section{Derivation of the Macro-Model}
\label{sec:derivation}

In this section we derive a macroscopic PDE-based model for the SPPs and the obstacles. The IBM model in \eqref{eqn:IBMmodel_massless} serves as the starting point. The derivation is a two-step process: First we formally derive a kinetic description for both the SPPs and the obstacles by taking a mean-field limit. In the second step we use a hydrodynamic scaling for the SPPs and derive equations for the SPP density and orientation. For this step we use previous work \cite{Degond2008, Degond2015}. For the obstacles we focus on a particular parameter regime and assume to have low obstacle noise and strong obstacle spring stiffness. The main technical difficulty and new derivation strategy lies in this last step. Fig.~\ref{fig:deriv_overview} summarises the different derivation steps. Throughout the document the domains of integrations are understood to mean the whole domain, unless specified otherwise.

\begin{figure}[t]
\centering
\includegraphics[width=0.8\textwidth]{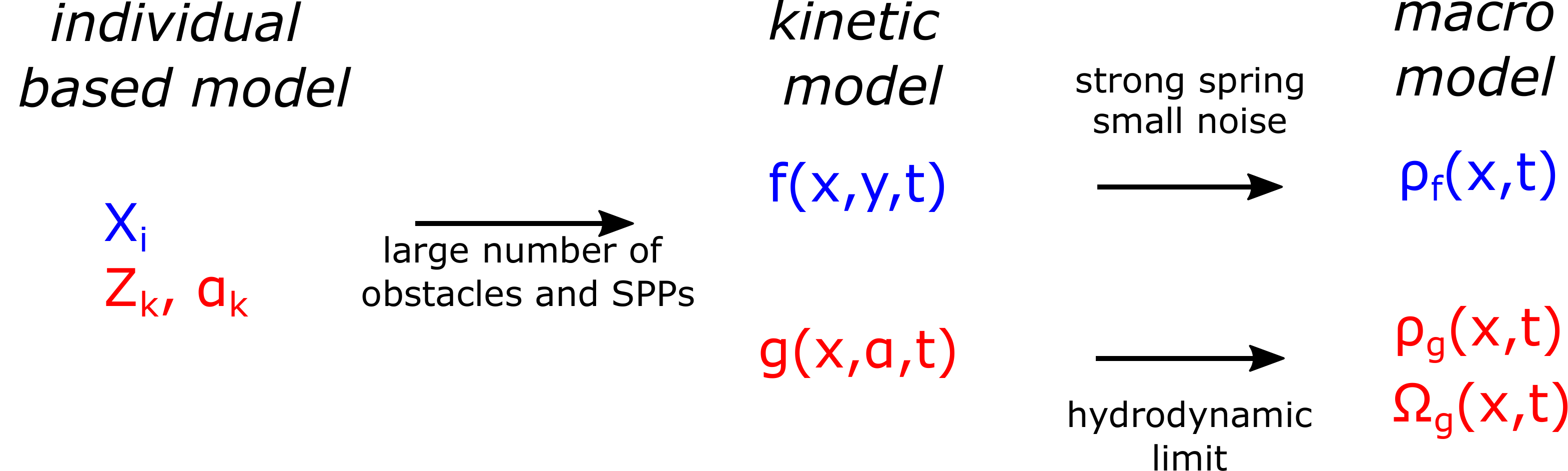}
\caption{\small{Derivation overview. Shown are the different levels of models used in this work and the connection between them. See text for further details.}}
\label{fig:deriv_overview}
\end{figure}

\subsection{The mean-field limit}
\label{ssec:meanField}

We start by defining $g( x, \alpha,t)$, the distribution of the SPPs at position $x \in \R^n$, time $t \geq 0$ with direction of the self-propelled velocity $\alpha \in \mathbb{S}^{n-1}$ and let $f( x, y, t)$ be the distribution of obstacles with position $x \in \R^n$, tethered at $y \in \R^n$ at time $t \geq 0$.\newline\par

We consider the empirical distribution associated with the dynamics of the SPPs and tethered obstacles given by system \eqref{eqn:IBMmodel_massless}.
\begin{align}
\label{eqn:empDist}
 g^M ( x, \alpha, t) & = \frac{ 1}{ M} \sum_{k=1}^M \delta_{Z_k(t)} ( x) \otimes \delta_{ \alpha_k(t)} ( \alpha) \, , \\
 f^N ( x, y, t) & = \frac{ 1}{ N} \sum_{i=1}^N \delta_{X_i(t)} ( x) \otimes \delta_{ Y_i(t)} ( y),\nonumber
\end{align}
where $\delta_{A}$ denotes the Dirac delta in $\R^n$ (for $A=X_i, Y_i, Z_k$) or in $\S^{n-1}$ (for $A=\alpha_k$) concentrated at $A$.
 

\begin{lemma}[Kinetic Model] Formally, as $N, M \rightarrow \infty$, $f^N \rightarrow f$ and $g^M \rightarrow g$, where the distributions $f(x,y,t)$ and $g(x,\alpha,t)$ fulfil the following Kolmogorov-Fokker-Planck equations
\begin{subequations}
\label{eqn:kinModelMassless}
\begin{align}
& \dt f + \nabla_x \cdot \big( \mathcal{W} f \big) = d_o \Delta_x f, \label{eqn:kinModelMassless_f}\\
& \dt g + \nabla_x \cdot \big( \mathcal{U} g \big) +
\nabla_\alpha \cdot \bigg( P_{\alpha^\perp} \left[ \nu \bar{\alpha}_g \right] g \bigg) = d_s \Delta_\alpha g,\label{eqn:kinModelMassless_g}
\end{align}
\end{subequations}
where
\begin{align}
\label{eqn:BarAlpha}
&\bar{ \alpha}_g(x,t) = \frac{ J_g ( x, t)}{| J_g ( x, t)|} \qquad \text{ with} \qquad J_g ( x, t) = \int_{|x-z|\leq r_A}\!\!\! \alpha\, g ( z, \alpha, t) \dd z \dd \alpha.
\end{align}
For the (space and time dependent) velocities we have
\begin{align}
\label{eqn:speeds}
\mathcal{W}&= - \frac{\kappa}{\eta} ( x - y) - \frac{1}{\eta} \nabla_x \bar{\rho}_g(x,t),\\
\mathcal{U} &=  \alpha - \frac{ 1}{ \zeta} \nabla_x \bar{\rho}_f(x,t)- \frac{ 1}{\zeta} \nabla_x \hat{\rho}_g(x,t),\nonumber
\end{align}
where we have introduced the densities of obstacles and SPPs
\begin{equation}
\label{eqn:densities}
\rho_g(x,t)=\int g ( x, \alpha, t) \dd \alpha,\qquad \rho_f(x,t)=\int f( x, y, t) \dd y,
\end{equation}
as well as an abbreviation for densities convoluted with kernels
\[
\bar \rho(x,t):=(\phi * \rho)(x,t), \qquad \hat \rho(x,t):=(\psi * \rho)(x,t).
\]
Further $f$ fulfils
\begin{equation}
\label{eqn:tethDens}
\int f(x,y,t) \dd x = \rho_A(y),
\end{equation}
where $\rho_A (y)$ is a given, time-independent function of obstacle anchor positions.
\end{lemma}
\begin{proof}
The limit is purely formal and uses standard techniques. We observe that $f^N$ and $g^M$ fulfil the equations for all $N$ and $M$ and then pass to the limit.
\end{proof}

\begin{remark}
 Note that since $f$ and $g$ are probabilities they also fulfil
\[
\int f(x,y,t) \dd x \dd y = \int g(x, \alpha, t) \dd x \dd \alpha \equiv 1,
\]
and consequently
\[
\int \rho_A (y) \dd y = 1.
\]
\end{remark}

\paragraph{Interpretation.}
At this point we have a system of coupled kinetic equations for the obstacle distribution $f(x,y,t)$ and the SPP distribution $g(x,\alpha,t)$. The interactions between the obstacles and the SPPs lead to the terms of the form $\nabla_x\bar\rho$ in the speeds $\mathcal{W}$ and $\mathcal{U}$ in \eqref{eqn:speeds}. An easy way to understand these terms is by assuming that the interaction force is of repulsive nature and purely local, in which case $\nabla_x\bar\rho=\nabla_x \rho$. We then see that the interaction force moves obstacles and SPPs in the opposite direction of the gradient of each other. The convolution with $\phi$ accounts for the potential non-locality of this interaction, which will be crucial later on. The remaining terms in $\mathcal{W}$ and $\mathcal{U}$ show the influence of the tethers and the self-propulsion for obstacles and the SPPs respectively. In $\mathcal{U}$ we also see the influence of SPP repulsion. The term involving $\nabla_\alpha$ in \eqref{eqn:kinModelMassless_g} reflects the effect of SPP alignment. The terms on the right-hand-side of \eqref{eqn:kinModelMassless} are results of the stochasticity in the obstacle position (for $f$) and in the SPP orientation (for $g$).

\subsection{Scaling assumptions}
\label{ssec:scaling}

To derive the macroscopic equations for the SPP-obstacle interactions we make a number of scaling assumptions for both the SPPs and the obstacles.

\paragraph{Scaling assumptions for the SPPs.} Following previous work \cite{Degond2008, Degond2015}, we introduce a small parameter $\eps$ and specify the relative order of the various terms. We mostly follow \cite{Degond2015}, with a few small differences: Firstly we assume the effect of alignment to be purely local, i.e. $r_A=\mathcal{O}(\eps)$, as has been done e.g. in \cite{Degond2008}. Alternatively one could choose a weakly non-local scaling $r_A=\mathcal{O}(\sqrt\eps)$, which would lead to an additional viscous term in the SPP orientation equation \eqref{eqn:macroSw_Omega} below. As in \cite{Degond2015} we also assume the SPP self-repulsion to be purely local, i.e. $r_R=\mathcal{O}(\eps)$ and assume that
$$
\int \psi(x)\dd x=:\mu<\infty.
$$
However we do not make any smallness assumption with regards to the SPP-obstacle interaction scale $r_I$. This is because we are interested in studying the effect of the non-locality of this interaction. Otherwise we proceed as in \cite{Degond2015}, i.e. assuming the alignment frequency $\nu$ and orientational diffusion $d_s$ to be of order $1/\eps$, and their ratio to be of order one.

\paragraph{Scaling assumptions for the obstacles.}

From \eqref{eqn:speeds} we see that it is only the macroscopic obstacle density $\rho_f(x,t)$ that enters the SPP equation. Unfortunately we cannot obtain a closed system for the macroscopic obstacle density $\rho_f(x,t)$ of $f$ by integrating \eqref{eqn:kinModelMassless_f}. Instead we make assumptions about the time scales of the obstacle dynamics. From now on we also assume to have a constant anchor density, i.e. $\rho_A \equiv 1$ is constant in space and time. We note that the results can be generalised to non-uniform $\rho_A$. We introduce the following quantities
\[
\gamma=\eta/\kappa,\qquad \delta =d_o \gamma.
\]
For the derivation we will assume both $\gamma$ and $\delta $ to be small. For $\gamma$ this means that the obstacle spring relaxation time scale is small compared to the SPP domain crossing time. We will sometimes refer to this assumption as `stiff obstacles', since it can be realized with a large spring constant $\kappa$. For $\delta $ smallness means that the obstacle spring relaxation time scale is small compared to the obstacle diffusion time scale, which we refer to as `low obstacle noise'.
Next we rewrite \eqref{eqn:kinModelMassless_f} as
\begin{align}
\label{eqn:kinModelMassless_f_limit}
\dt f + \nabla_x \cdot \left(\ve(x,t) f\right)=\frac{1}{\gamma} \mathcal{A}_y(f),
\end{align}
where we have defined the `external' velocity as
\begin{align}
\label{eqn:vext}
\ve(x,t)=-\frac{1}{\eta}\nabla_x\bar{\rho}_g(x,t)
\end{align}
and the operator $\mathcal{A}_y$ by
\begin{equation}
\label{eqn:operatorA}
\mathcal{A}_y(f):=\nabla_x \cdot \left[ (x-y)f + \delta  \nabla_x f \right].
\end{equation}
We can rewrite the operator as
\begin{align*}
\mathcal{A}_y(f)=\delta  \nabla_x \cdot \left[M_\delta (x-y)\nabla_x\left(\frac{f}{M_\delta (x-y)}\right)\right],
\end{align*}
where  $M_{\delta}(z)$ is a Gaussian with variance $\delta$ centred around 0, whose mass is normalized to one, i.e.
\begin{align}
\label{eqn:Gaussian}
M_{\delta}(z)=\frac{1}{Z_\delta}e^{-\frac{|z|^2}{2\delta}}, \qquad Z_\delta=(2\pi \delta)^{n/2}.
\end{align}
The above also shows that $M_\delta (x-y)$ is in the kernel of $\mathcal{A}_y$.
\begin{remark}
Note that the rescaling of the diffusion term $\delta =d_o\gamma$ ensures the operator $\mathcal{A}_y$ is a Fokker-Planck-type  operator. Without it, we would obtain $\mathcal{A}_y(f)=\nabla_x \cdot \left[ (x-y)f\right]$, whose kernel contains Dirac deltas, making the analysis much more tedious. Eventually, however, we are interested in the small noise limit. This, of course raises several questions, which are beyond the scope of this work, e.g. does the order of the the limits $\gamma \rightarrow 0$ and $\delta  \rightarrow 0$ matter?
\end{remark}

\subsection{The macroscopic SPP-obstacle equation}
\label{ssec:macroswimmer}

Using the scaling and notation above, we now state the main result of this section, which we prove in Sec.~\ref{ssec:proof}.

\begin{theorem}(SPP-Obstacle Macromodel)
\label{thm:macromodel}
Let $\rho_A\equiv 1$ be constant and $f(x,y,t)$ fulfill \eqref{eqn:kinModelMassless_f_limit} with $\gamma \ll 1$ and $\delta  \ll 1$. Further let $g^\eps(x,\alpha,t)$ be the solution of \eqref{eqn:kinModelMassless_g} using the scaling involving $\eps$ described above and let $g^0(x,\alpha,t)$ be its (formal) limit as $\eps\rightarrow 0$. Then it holds that
$$
g^0(x,\alpha,t)=\rho_g(x,t)N_{\Omega_g(x,t)}(\alpha),
$$
where  $N_{\Omega}$ is the von Mises-Fisher distribution defined by
$$
N_{\Omega}(\alpha)=\frac{1}{K_{d}}e^{\frac{\Omega \cdot \alpha}{d}},\qquad K_{d}=\int e^{\frac{\Omega \cdot \alpha}{d}} \dd \alpha, \qquad d=\frac{d_s}{\nu}, \qquad \text{for}\quad \Omega \in \mathbb{S}^{n-1}.
$$
Note that $K_d$ is a normalization constant and is independent of $\Omega$. Further the macroscopic SPP density $\rho_g(x,t)$ and the macroscopic SPP orientation $\Omega_g(x,t)$ fulfil
\begin{subequations}
\label{eqn:macroSwObs}
\begin{align}
& \dt \rho_g + \nabla_x \cdot \left( U \rho_g \right) = 0,\label{eqn:macroSw_rho}\\
& \rho_g \dt \Omega_g + \rho_g \left(V \cdot \nabla_x\right)\Omega_g +  d P_{\Omega_g^\perp} \nabla_x \rho_g
= 0,\label{eqn:macroSw_Omega}\\
& U=c_1\Omega_g-\frac{1}{\zeta}\nabla_x \bar{\rho}_f-\frac{\mu}{\zeta}\nabla_x \rho_g, \quad V=c_2\Omega_g-\frac{1}{\zeta}\nabla_x \bar{\rho}_f-\frac{\mu}{\zeta}\nabla_x \rho_g,\label{eqn:macroSw_UV}
\end{align}
\end{subequations}
The constants $c_1>0$ and $c_2>0$ depend only on $d=d_s/\nu$ and are defined as in \cite{Degond2015}. The macroscopic obstacle density $\rho_f(x,t)$ is given by
\begin{align}
\label{eqn:obsDensMax}
\rho_f(x,t)=1 &- \frac{\gamma}{\delta  \eta} \bigg [ \bar{\rho}_g(x)-\big [M_{2\delta }*\bar{\rho}_g\big](x) \bigg ] \nonumber \\
&-\frac{\gamma^2}{\eta}\pd_t \Delta_x \bar \rho_g +\frac{\gamma^2}{\eta^2}\mathcal{N}(\bar\rho_g)+\mathcal{O}(\gamma^2\delta )++\mathcal{O}(\gamma^3),
\end{align}
where the nonlinear term $\mathcal{N}$ is defined by
$$
\mathcal{N}(\bar\rho_g)=\frac{1}{2}\left[(\Delta_x \bar \rho_g)^2-\mathbb{H}(\bar\rho_g):\mathbb{H}(\bar\rho_g)\right] \, ,
$$
where $\mathbb{H}(\bar\rho_g)$ denotes the Hessian of the
function $\bar\rho_g$, i.e. $\{ \mathbb{H}(\bar\rho_g)\}_{i,j} = \partial_i \partial_j \bar\rho_g$, and
given two $n$ by $n$ matrices $\mathbb{A}$ and $\mathbb{B}$,
their scalar product is defined as $\mathbb{A}:\mathbb{B} = \sum_{i,j=1}^n A_{i,j} B_{i,j}$.
\end{theorem}

Eqs.~\eqref{eqn:macroSw_rho} and \eqref{eqn:macroSw_Omega} give the evolution for the particle density $\rho_g$ and mean orientation $\Omega_g$ respectively. Without the term $\nabla_x\bar\rho_f$ appearing in $U$ and $V$ in Eq.~\eqref{eqn:macroSw_UV} these equations correspond to the so-called Self-Organised Hydrodynamics with Repulsion (SOHR) and their derivation can be found in \cite{Degond2015}. The additional terms in Eq.~\eqref{eqn:macroSw_UV} account for the influence of the obstacles density $\rho_f$.\newline

The equation for the obstacle density, expanded in the small variables $\delta $ and $\gamma$ is given in \eqref{eqn:obsDensMax}. It is important to note that the obstacle density given in \eqref{eqn:obsDensMax} can in principle become negative, which is not physically meaningful. This is a consequence of the assumption that $\gamma$ is small and indicates that the validity of the model will be limited to certain parameter regimes.  We see that for infinitely strong springs, i.e. $\gamma\rightarrow 0$, $\rho_f(x,t)\equiv \rho_A\equiv 1$, i.e. obstacles remain exactly at their anchor points and since those are assumed to be uniformly distributed, the obstacles have no effect on the SPPs ($\nabla_x\bar\rho_f\equiv 0$). For small, but finite $\gamma$ the feedback from the SPPs leads to non-uniform obstacles. 

\paragraph{Influence of obstacle noise.} The influence of the obstacle noise $\delta $ is contained in the order $\gamma$ term in \eqref{eqn:obsDensMax}. We note that
\begin{align*}
- \frac{1}{\delta  \eta} \bigg [ \bar{\rho}_g(x)-\left(M_{2\delta }*\bar{\rho}_g\right)(x) \bigg ]\rightarrow \frac{1}{\eta}\Delta_x \bar{\rho}_g(x)\qquad \text{as}\quad \delta \rightarrow 0.
\end{align*}
We see that the noise adds an additional form of non-locality. Whether the obstacle density is reduced or increased depends on whether $\bar{\rho}_g$, the convoluted SPP density at $x$ is smaller or larger than the `blurred', convoluted SPP density $\bar{\rho}_g$, where the amount of blurring depends on the obstacle noise. In the absence of obstacle noise \eqref{eqn:obsDensMax} simplifies to
\begin{align}
\label{eqn:obsDens_noNoise}
\rho_f(x,t)=1 +\frac{\gamma}{\eta}\Delta_x \bar{\rho}_g(x)-\frac{\gamma^2}{\eta}\pd_t \Delta_x \bar \rho_g +\frac{\gamma^2}{\eta^2}\mathcal{N}(\bar\rho_g)+\mathcal{O}(\gamma^3).
\end{align}

\paragraph{SPP dynamics deform obstacle volume elements.} 

In the absence of obstacle noise we can rewrite \eqref{eqn:obsDens_noNoise} as
\begin{align}
\label{eqn:densObsDet}
\rho_{f}(x)&=\det{J_Y}-\frac{\gamma^2}{\eta}\pd_t \Delta_x \bar \rho_g+\mathcal{O}(\gamma^3),
\end{align}
where $J_Y$ is the Jacobian of the map
\begin{align*}
&Y(x,t)=x+\frac{\gamma}{\eta}\nabla_x\bar{\rho}_g(x,t).
\end{align*}
The map $Y$ can be interpreted as an estimate of the anchor position of an obstacle at position $x$ moved under the influence of the SPP density. Then the determinant of the Jacobian reflects the deformation of a volume element of obstacles due to the SPPs. Note that for $n=3$ $\det{J_Y}$ contains also order $\gamma^3$ terms, for $n=2$ only order $\gamma^2$ terms and lower.

\paragraph{Higher order terms account for SPP movement.} Finally we comment on the time derivative appearing in \eqref{eqn:obsDensMax}. The time derivative leads to a form of delay, i.e. the obstacles retain a memory of where SPPs were. This can be seem by Taylor expanding the SPP density in time using the time scale of obstacle relaxation $\gamma$. Then the linear terms in \eqref{eqn:obsDens_noNoise} can be written as
\begin{align*}
\frac{\gamma}{\eta}\Delta_x\left(\bar\rho_g-\gamma\pd_t\bar\rho_g\right)=\frac{\gamma}{\eta}\Delta_x \bar\rho_g(x,t-\gamma)+\mathcal{O}(\gamma^2).
\end{align*}

Finally in preparation for the analytical and numerical investigation of Sec.~\ref{sec:1dAnalysis} and Sec.~\ref{sec:numerics} we state the following:


\begin{corollary}[1D equations.] Let the assumptions of Thm.~\ref{thm:macromodel} hold. Then for $n=1$, the equations for the SPP density $\rho_g(x,t)$ and the obstacle density $\rho_f(x,t)$ with $x\in \mathbb{R}$ and $t\geq 0$ are given by
\begin{align}
\label{eqn:macroSw_rho_1D}
&\pd_t\rho_g+c_1\pd_x \rho_g=\frac{1}{\zeta}\left(\mu \rho_g\pd_x\rho_g+\rho_g\pd_x \bar\rho_f\right),
\end{align}
where we have assumed all particles move to the right. The obstacle density up to order $\gamma^2$ is given by
\begin{align}
\label{eqn:obsDensMax_1D}
\rho_f(x,t)=1 &- \frac{\gamma}{\delta  \eta} \bigg [ \bar{\rho}_g(x)-\big [M_{2\delta }*\bar{\rho}_g\big](x) \bigg ]-\frac{\gamma^2}{\eta}\pd_t \pd_x^2 \bar \rho_g.
\end{align}
For $\delta \rightarrow 0$ and using only terms up to order $\gamma$, \eqref{eqn:obsDensMax_1D} simplifies to
\begin{align}
\label{eqn:obsDensSimple_1D}
\rho_f(x,t)=1 + \frac{\gamma}{\eta}\pd_x^2\bar\rho_g.
\end{align}
\end{corollary}

\subsection{Proof of Theorem 1}
\label{ssec:proof}

For the coarse-graining of the kinetic SPP equation \eqref{eqn:kinModelMassless_g} we refer to previous work \cite{Degond2008, Degond2015}. We note that the obstacle density enters the SPP equation solely through its macroscopic density $\rho_f(x,t)$ via the interaction operator $\nabla_x\bar{\rho}_f$, which has a structure analogous to the SPP self-repulsion term, hence analogous techniques can be applied. \newline

To derive an expression for the obstacle density $\rho_f(x,t)$, we formulate and proof the following Theorem:

\begin{theorem}
\label{thm:limit}
Let $\rho_A\equiv 1$ and $f(x,y,t)$ fulfil \eqref{eqn:kinModelMassless_f_limit} with $\mathcal{A}_y(f)$ defined in \eqref{eqn:operatorA}. Let $\gamma \ll 1$ and expand $f(x,y,t)$ as
\begin{equation}
f(x,y,t)=f_0(x,y,t)+\gamma f_1(x,y,t)+\gamma^2 f_2(x,y,t)+\mathcal{O}(\gamma^3).
\end{equation}
Then the macroscopic densities defined by 
\begin{equation}
\rho_{f_i}(x,t)=\int f_i(x,y,t) \dd y
\end{equation}
satisfy
\begin{align}
\label{eqn:0thMoment}
\rho_{f_0}(x)&=1\\
\rho_{f_1}(x)&=-{\rm div} (\ve)-\delta \frac{1}{2}\Delta_x  {\rm div} (\ve)+ \mathcal{O}(\delta ^2),\nonumber\\
\rho_{f_2}(x)&=\frac{1}{2}\nabla_x \cdot \left[\ve\, {\rm div}{(\ve)}-(\ve \cdot \nabla_x)\ve\right]+\pd_t {\rm div}{(\ve)}+ \mathcal{O}(\delta ),\nonumber
\end{align}
as $\delta \rightarrow 0$.
\end{theorem}

\begin{proof}
In the following we drop the $t$-dependence of most terms to increase readability. We obtain the following equations for the three highest orders of $\gamma$
\begin{subequations}
\begin{align}
    &\mathcal{A}_y(f_0)=0,\label{eqn:order0} \\
    & \mathcal{A}_y(f_1)=\dt f_0+\nabla_x \cdot (\ve(x)f_0),\label{eqn:order1}\\
    & \mathcal{A}_y(f_2)=\dt f_1 +\nabla_x \cdot (\ve(x)f_1).\label{eqn:order2}
\end{align}
\end{subequations}

Let us note that \eqref{eqn:order0}, \eqref{eqn:order1}, and \eqref{eqn:order2} can be recast as follows: Given a function $h$ 
find $\psi$ (in a suitable functional space) such that
\begin{equation}\label{eqn:fredholm}
\mathcal{A}_y( \psi) = h \, . 
\end{equation}
Due to the conservation of mass property of $\mathcal{A}_y$, i.e. $\int \mathcal{A}_y \dd x = 0$,
a necessary conditon to warranty the existence of a solution of \eqref{eqn:fredholm} is $\int h \dd x = 0$. 
It can be shown that the operator $\mathcal{A}_y$ has compact resolvent on a suitable functional space and its kernel is generated by $M_{\delta }(x-y)$, given in \eqref{eqn:Gaussian}. The most important properties of the Gaussian $M_\delta $, that we will use repeatedly, are
\[
\int M_{\delta}(z) \dd z=1, \qquad \int z M_{\delta}(z) \dd z=0, \quad \nabla_z M_{\delta}(z)=-\frac{z}{\delta}M_{\delta}(z).
\]
Hence, we can obtain a complete characterization of the solutions 
of \eqref{eqn:fredholm} via the Fredholm alternative, namely, for any function $h$ such that $\int h \dd x = 0$
there exists a unique solution $\psi$ up to an element of the kernel of $\mathcal{A}_y$. For a proof of this result consult \cite{AS2019}.\newline

Let us start by considering \eqref{eqn:order0}, we search for a solution $f_0$ such that $\int f_0 ( x, y) \dd x = 1$, hence,
according to the results obtained for \eqref{eqn:fredholm} the unique solution is given as
\begin{align}
\label{eqn:order0_sol}
f_0(x,y)= M_{\delta }(x-y),
\end{align}
where $M_{\delta }$ is defined in \eqref{eqn:Gaussian}. For the remaining two equations we require the 
following scaling condition to hold, which ensures that the average mass is one,
\begin{align}
\label{eqn:scaling}
   \int f_i(x,y,t)\dd x=0,\qquad i=1,2.
\end{align}

\paragraph{Step 1: Rescaling.}

 Next we define the functions $h_1(\sigma, y,t)$ and $h_2(\sigma, y,t)$ as
\begin{align*}
f_1(x,y,t)&=\frac{1}{\sqrt{\delta }}\, M_\delta (x-y)\,h_1\left(\frac{x-y}{\sqrt{\delta }},y,t\right),\\
f_2(x,y,t)&=\frac{1}{\delta }\, M_\delta (x-y)\,h_2\left(\frac{x-y}{\sqrt{\delta }},y,t\right).
\end{align*}
This turns \eqref{eqn:order1} and \eqref{eqn:order2} into equations for $h_1(\sigma,y,t)$ and $h_2(\sigma,y,t)$. 
Defining $\mathcal{B}$ as the operator
\begin{align}
\label{eqn:FokkerPlanck}
&\mathcal{B}(h)=\Delta_\sigma h-\sigma \cdot \nabla_\sigma h,
\end{align}
we obtain, after tedious but straightforward computations, the following relationships
\begin{align}
\label{eqn:hSystem}
\mathcal{B}(h_1)&=\sqrt{\delta }\,\Div(\ve)|_{y+\sqrt{\delta }\sigma}-\sigma\cdot \ve|_{y+\sqrt{\delta }\sigma},\\
\mathcal{B}(h_2)&=\sqrt{\delta }\left(\dt h_1 + h_1\,\Div(\ve)|_{y+\sqrt{\delta }\sigma}\right)+\ve|_{y+\sqrt{\delta }\sigma}\cdot (\nabla_\sigma h_1-\sigma\cdot h_1).\nonumber
\end{align}
We use the notation $\Div = \nabla_x \cdot$ for the divergence of a vector field. 
There are several advantages to this scaling: Firstly, the operator $\mathcal{B}$ is the generator
of the Ornstein-Uhlenbeck stochastic process (a consequence of using $\sigma=(x-y)/\sqrt{\delta }$) and we can use its well 
known properties directly without having to scale by $\delta $. Secondly, we have removed the 
Gaussian $M_{\delta }$ from the equation (it cancelled). 
Finally, additionally scaling $f_1$ and $f_2$ by $1/\sqrt{\delta }$ and $1/\delta $ respectively turns out to be the correct choice when calculating the densities.\newline\par
\noindent
Before we proceed to the next step, we need to collect a number of properties of $\mathcal{B}$, 
all of which are well known and stated in App.~\ref{ssec:FPO}.

\paragraph{Step 2: Expansion in terms of the obstacle noise $\delta $.}
The next step involves expansion of the right-hand-sides of \eqref{eqn:hSystem}, $h_1$ and $h_2$ with respect to $\delta $, i.e.
\begin{align*}
h_1(\sigma,y,t)&=h_1^0(\sigma,y,t)+\sqrt{\delta }h_1^1(\sigma,y,t)+\delta  h_1^2(\sigma,y,t)+\mathcal{O}(\delta ^{3/2}),\\
h_2(\sigma,y,t)&=h_2^0(\sigma,y,t)+\sqrt{\delta }h_2^1(\sigma,y,t)+\delta  h_2^2(\sigma,y,t)+\mathcal{O}(\delta ^{3/2}).
\end{align*}
This yields as equations for $h_1^0$, $h_1^1$ and $h_1^2$
\begin{align*}
\mathcal{B}(h_1^0)&=-\ve_k \sigma_k,\\
\mathcal{B}(h_1^1)&=\pd_i \ve_i- \sigma_k \sigma_i \pd_k \ve_i, \\
\mathcal{B}(h_1^2)&=\sigma_k \pd_{ki}\ve_i - \frac{1}{2}\sigma_k \sigma_i \sigma_j \pd_{ij}\ve_k.
\end{align*}
Note that we have used the Einstein's summation convention and that now $\ve$ and its derivatives are all evaluated at $(y,t)$. 
Here partial derivatives are understood to act on the spatial variable, i.e. $\pd_i \ve := \frac{\pd}{\pd y_i} \ve(y,t)$. The advantage of this procedure is the following: Now the right hand sides are low-order polynomials in $\sigma$ and since $\mathcal{B}$ only acts on $\sigma$, the equations can be solved explicitly by rewriting the right hand sides in terms of the Hermite basis and using P2 of Lem.~\ref{lem:FokkerPlanck} in App.~\ref{ssec:FPO}. This procedure yields the explicit solutions
\begin{align}
\label{eqn:h1sol}
h_1^0(\sigma,y,t)&=\ve_i\H_{e_i},\\
h_1^1(\sigma,y,t)&=\frac{1}{2}\pd_k \ve_j \H_{e_k+e_j},\nonumber \\
h_1^2(\sigma,y,t)&=\frac{1}{2}\left[\pd_{ii}\ve_k\, \H_{e_k}+\frac{1}{3}\pd_{ij}\ve_k\, \H_{e_k+e_i+e_j}\right],\nonumber
\end{align}
where $\H$ are the tensor Hermite polynomials defined in Lem.~\ref{lem:FokkerPlanck} in App.~\ref{ssec:FPO}. Note that $\ve$ and all its derivatives are evaluated at $(y,t)$ and $\H$ at $\sigma$.

\vspace{0.4cm}

\noindent As equations for $h_2^0$ and $h_2^1$ we obtain
\begin{align*}
\mathcal{B}(h_2^0)&=\ve \cdot \left(\nabla_\sigma h_1^0-\sigma h_1^0\right),\\
\mathcal{B}(h_2^1)&=\dt h_1^0+\pd_i \ve_i\,h_1^0 +(\sigma_k\pd_k\ve) \cdot \left(\nabla_\sigma h_1^0-\sigma h_1^0\right)+ \ve \cdot \left(\nabla_\sigma h_1^1-\sigma h_1^1\right).
\end{align*}
As above $\ve$ and its derivatives are all evaluated at $(y,t)$. Using the solutions for $h_1^0$, $h_1^1$ and $h_1^2$ given in \eqref{eqn:h1sol} we can solve the equations for $h_2^0$ and $h_2^1$ in the same fashion, yielding the explicit expressions
\begin{align}
\label{eqn:h2sol}
h_2^0(\sigma,y,t)&=\frac{1}{2}\ve_k\ve_j \H_{e_k+e_j},\\
h_2^1(\sigma,y,t)&=\left(-\dt\ve_k + \ve_i \pd_i \ve_k\right)\, \H_{e_k}+\frac{1}{2}\ve_i\pd_k\ve_j\, \H_{e_k+e_i+e_j}\nonumber.
\end{align}
Note that the solutions fulfil the scaling condition \eqref{eqn:scaling} since it holds that
\begin{align}
\label{eqn:scaling2}
   \int M_1(\sigma)h_i^j(\sigma,y,t)\dd \sigma=0,\quad i=1,2,\quad j=0,1,2.
\end{align}


\paragraph{Step 3: Calculating the macroscopic moments of the obstacle density.}
With the preparation of the two steps above, the calculation of the obstacle densities
\begin{equation*}
\rho_{f_i}(x,t)=\int f_i(x,y,t) \dd y,
\end{equation*}
and consequently its contribution to the SPP equation becomes relatively simple. The procedure and calculations are described in App.~\ref{ssec:moments}. This yields \eqref{eqn:0thMoment} as claimed.
\end{proof}


\paragraph{Explicit solution for $f_1$.}
\label{ssec:consVF}


The above outlined procedure works for any given external velocity $\ve(x,t)$, i.e. it allows to include other influences as well. 
For example, in future work we plan to use the derivation strategy to include the description of a fluid in which the obstacles and 
SPPs are immersed in. However, for this model, we can use the fact that $\ve(x,t)$ is in fact a conservative vector field. This allows to solve the 1-st order equation \eqref{eqn:order1} for $f_1(x,y,t)$ directly. This is covered in the following lemma, where the $t$ dependence has been suppressed for notational convenience. 

\begin{lemma}
\label{lem:consVF}
Let $\ve(x)$ be a conservative vector field, i.e. there exists a scalar function $V(x)$, such that $\nabla_x V=\ve$, then we can write the solution to \eqref{eqn:order1} as
\begin{align*}
f_1(x,y)=M_{\delta }(x-y)\frac{1}{\delta }\left[V(x)-\left(M_{\delta }*V\right)(y)\right]
\end{align*}
\end{lemma}
\begin{proof} By direct calculation we see that
\begin{align*}
&(x-y)f_1+\delta  \nabla_x f_1 =  M_\delta  (x-y)\ve(x),
\end{align*}
which shows that $f_1$ is indeed a solution to \eqref{eqn:order1}. Finally we have to verify the normalization condition \eqref{eqn:scaling}
\begin{align*}
\int f_1(x,y) \dd x&=\frac{1}{\delta }\int M_{\delta }(x-y)\left[V(x)-\left(M_{\delta }*V\right)(y)\right]\dd x\\
&=\frac{1}{\delta }\left[\int M_{\delta }(x-y)V(x)\dd x-\left(M_{\delta }*V\right)(y)\right]=0,
\end{align*}
which finishes the proof.
\end{proof}



The above Lemma is applicable for this model of SPP-obstacle interactions since we have that 
\[
\ve(x,t)=-\frac{1}{\eta}\nabla_x \bar{\rho}_g(x,t),
\]
i.e. we can use Lem.~\ref{lem:consVF} with $V(x,t)=-\frac{1}{\eta} \bar{\rho}_g(x,t)$. We consequently find
\begin{align*}
f_1(x,y,t)&=- M_{\delta }(x-y)\frac{1}{\delta  \eta}\left[\bar{\rho}_g(x,t)-(M_{\delta }*\bar{\rho}_g)(y,t)\right].
\end{align*}
From this we can calculate 
\begin{align}
\label{eqn:obstDens_nonLoc}
\rho_{f_1}(x,t)&=- \frac{1}{\delta  \eta} \bigg [ \bar{\rho}_g(x)-\left(M_{2\delta }*\bar{\rho}_g\right)(x) \bigg ].
\end{align}

\begin{remark}
Note that since
\begin{align*}
- \frac{1}{\delta  \eta} \bigg [ \bar{\rho}_g(x)-\left(M_{2\delta }*\bar{\rho}_g\right)(x) \bigg ]&=\frac{1}{\eta}\left[\Delta_x \bar{\rho}_g(x)+\frac{\delta }{2}\Delta_x^2 \bar{\rho}_g\right]+\mathcal{O}(\delta ^2),
\end{align*}
we see that this is consistent with \eqref{eqn:0thMoment}, but contains more information about the $\mathcal{O}(\delta ^2)$ term.
\end{remark}

\paragraph{The macroscopic obstacle density.}

Collecting the results of Thm.~\ref{thm:limit} and Lem.~\ref{lem:consVF} and using the definition of $\ve$ given in \eqref{eqn:vext} we find that the maximum order of approximation of the obstacle density we can now write is given in \eqref{eqn:obsDensMax} as claimed. This finishes the proof of Thm.~\ref{thm:macromodel}.


\section{Analytical insights from the 1D Macromodel.}
\label{sec:1dAnalysis}

In this section we analyse the macroscopic model derived in  Sec.~\ref{sec:derivation} further to gain insights into the SPP-obstacle interactions. In particular we use linear stability analysis to understand the onset of patterning and investigate how obstacles induce an effective SPP interaction.

\subsection{Linear Stability Analysis}
\label{ssec:stabAnalysis}

In this section we investigate pattern formation for the SPP-obstacle model. We work in one space dimension, i.e. we focus on the SPP density $\rho_g(x,t)$ and the obstacle density $\rho_f(x,t)$ for $x\in \mathbb{R}$ or and $t\geq 0$, whose dynamics are given by \eqref{eqn:macroSw_rho_1D} and \eqref{eqn:obsDensMax_1D}.\newline

Consider the steady state solution $\rho_g(x,t)=\rho_0>0$. Small perturbations of this solutions (called again $\rho_g$) then fulfil the linearised equation
\begin{align}
\label{eqn:swimmer1D_lin}
& \dt \rho_g + c_1 \pd_x \rho_g = \frac{\rho_0}{\zeta} \left( \mu \pd^2_x \rho_g+\pd^2_x \bar{\rho}_f  \right),
\end{align}
where $\rho_f$ is still given by \eqref{eqn:obsDensMax_1D}.\newline

The following propositions examines the growth or decay behaviour of perturbations of the constant steady state in dependence on their angular frequency and the resulting linear stability of the constant steady state. We consider the equation on the whole space $x\in \mathbb{R}$ and posed on an interval with periodic boundary conditions.


\begin{proposition}[Linear stability]
\label{prp:linstab}
 Consider \eqref{eqn:swimmer1D_lin} coupled to \eqref{eqn:obsDensMax_1D} posed a) on $x\in \mathbb{R}$ and b) on $x\in[0,1]$ with periodic boundary conditions.\newline
(i) The system permits solutions of the form $\rho_g(x,t)=\rho e^{i k x + \alpha t}$, with $\rho\neq 0$, $\alpha \in \mathbb{C}$ and $k\in \mathbb{R}$ (case a) or $k\in 2\pi \mathbb{Z}$ (case b) where $\alpha$ and $k$ fulfil the following dispersion relation
\begin{align}
\label{eqn:dispersion}
\alpha(k) = -i \frac{k c_1}{1+\gamma^2\frac{\rho_0}{\eta\zeta}k^2 \hat\phi_k^2}+\frac{\rho_0}{\zeta}k^2 \frac{\frac{\gamma}{\eta\delta }\left(1-e^{-\delta  k^2}\right)\hat\phi_k^2-\mu}{1+\gamma^2\frac{\rho_0}{\eta\zeta}k^2 \hat \phi_k^2},
\end{align}
where $\hat \phi_k$ is the Fourier transform (case a) or Fourier coefficient (case b) of the kernel $\phi$, defined by
\begin{align*}
\hat \phi_k=\int e^{-ikx}\phi(x) \dd x,
\end{align*}
where the integration domain is understood to be $\mathbb{R}$ (case a) or $[0,1]$ (case b).\newline
(ii) The constant steady state $\rho_g(x,t)=\rho_0$ is linearly stable iff
\begin{align}
\label{eqn:linstab}
\max_{k\in K} \frac{1}{\delta }\left(1-e^{-\delta  k^2}\right)\hat\phi_k^2<\frac{\mu \eta}{\gamma},
\end{align}
where $K=\mathbb{R}$ (case a) or $K=2\pi\mathbb{Z}$ (case b).

\end{proposition}
\begin{proof}
We show the proof for case a, case b can be shown analogously.
(i) Substituting the ansatz $\rho_g(x,t)=\rho e^{k i x + \alpha t}$ into \eqref{eqn:swimmer1D_lin} is equivalent to applying the Fourier transform to the whole equation. We use the following properties of the Fourier transform
\begin{align*}
\widehat{f*g}=\hat f \hat g,\qquad \widehat{\pd_x f}=i k\hat f, \qquad \widehat{M_\delta }=e^{-\frac{\delta }{2}k^2},
\end{align*}
and obtain an equation for $\hat \rho_g(k,t)=\int e^{-k i x}\rho_g(x,t) \dd x$.
\begin{align*}
\pd_t \hat \rho_g=-i k  c_1 \hat \rho_g -\frac{\rho_0}{\zeta}k^2 \left(\mu \hat \rho_g + \hat \phi_k \hat \rho_f\right).
\end{align*}
For the Fourier transform of $\rho_f$ we obtain
\begin{align*}
\hat \rho_f(k,t)=\tilde \delta(k)-\frac{\gamma}{\eta}\left(1-e^{-\delta  k^2}\right)\hat \phi_k \hat \rho_g+\frac{\gamma^2}{\eta}k^2 \hat \phi_k \pd_t \hat \rho_g,
\end{align*}
where $\tilde \delta(k)$ is the Dirac delta. Substituting $\hat\rho_f$ into the equation for $\hat \rho_g$ gives
$$
\pd_t \hat \rho_g(k,t)=\alpha(k)\hat \rho_g(k,t),
$$
with $\alpha(k)$ given in \eqref{eqn:dispersion} as claimed.\newline
(ii) We note that the decay or growth behaviour is determined by the sign of the real part of $\alpha(k)$. Since the denominator will always be positive, it is sufficient to examine the numerator. This gives the result.
\end{proof}

\begin{corollary}
\label{cor:linstab} Let $\rho_f$ be given only up to order $\gamma$ and let $\delta\rightarrow 0$. Then the real part of $\alpha(k)$ in Prop.~\ref{prp:linstab} becomes
\begin{align*}
\Re \alpha(k) =\frac{\rho_0}{\zeta}k^2 \left(\frac{\gamma}{\eta}(k\hat\phi_k)^2-\mu\right).
\end{align*}
\end{corollary}

\paragraph{Interpretation.} We interpret the results of Prop.~\ref{prp:linstab}(ii) as indication under what conditions patterning is expected. We start by observing that in the absence of obstacle noise, $\delta \rightarrow 0$, the linear stability condition \eqref{eqn:linstab} simplifies to
\begin{align*}
\max_{k} (k\hat\phi_k)^2<\frac{\mu \eta}{\gamma}.
\end{align*}
Since $\frac{1}{\delta }\left(1-e^{-\delta  k^2}\right)\leq k^2$, we observe that the obstacle noise $\delta>0$ has a stabilizing effect. The constant on the right-hand-side is critical for (in)stability. We see that SPP self-repulsion, strong obstacle springs and high obstacle friction stabilise the system. The order $\gamma^2$ approximation of the obstacle density leads to the additional terms in the denominator. It does not influence whether the constant steady state destabilises, however it decreases the growth or decay rate of the perturbations. The main determinant for pattern formation is the SPP-obstacle interaction kernel $\phi$ and the decay behaviour of its Fourier transform or coefficients. In case of purely local interactions $\hat \phi_k$ is constant and we see that the we have destabilisation for all parameter values, since for large frequencies the real part of $\alpha$ will always become positive. This emphasises the importance of the non-locality of the SPP-obstacle interactions. Next we look at a specific case.

\begin{example}
\label{ex:linStab} We assume $\delta\rightarrow 0$ and further consider the obstacle density only up to order $\gamma$. We work on $x\in[0,1]$ with periodic boundary conditions. Further we let the microscopic SPP-obstacle interaction kernel $\phi$ be compactly supported on the interval $[- r_I,  r_I]$ and yield a pushing force that decreases linearly with distance and is continuous at $ r_I$, i.e.
\begin{align*}
\phi(x)=\begin{cases}
C \frac{3}{2 r_I}\left(1-\frac{|x|}{ r_I}\right)^2 & \mbox {if } |x|< r_I \\
0 & \mbox{else}.\\
\end{cases}
\end{align*}
In this case we can calculate the Fourier coefficients explicitly and obtain
\begin{align*}
\hat \phi_k=6 C\frac{ r_I k -\sin{( r_I k)}}{( r_I k)^3}.
\end{align*}
The function
\begin{align*}
F(k)=(k\hat\phi_k)^2=\left(\frac{6 C}{ r_I}\right)^2\left(\frac{ r_I k -\sin{( r_I k)}}{( r_I k)^2}\right)^2
\end{align*}
attains its maximum at $k=\pi/ r_I$  and we hence find that if
\begin{align*}
\left(\frac{6 C}{\pi  r_I}\right)^2 < \frac{\mu \eta}{\gamma}.
\end{align*}
then the spatially constant steady state is linearly stable. The converse is in general not true, since $\pi/ r_I$ will typically not be in $2\pi\mathbb{Z}$. In the case of destabilisation, we expect the pattern size $P$ to be related to the maximum of $\Re \alpha(k)$, given in Cor.~\ref{cor:linstab}. We observe that $F(k)\rightarrow 0$ for $k \rightarrow \infty$ and hence $\Re \alpha(k)<0$ for $k$ sufficiently large. This means high-frequency perturbations will be damped by the diffusion-like SPP self-repulsion term. Since $\Re \alpha(0)=0$ there will typically be a well defined maximum attained at some $k=2\pi l_\text{max}$ with $l_\text{max}\in \mathbb{Z}$. We then expect that $P$ defined by 
$$P=\frac{1}{l_\text{max}}$$
will be a good indication of the expected pattern size. We numerically investigate whether this holds also far away from the constant steady state and for the IBM below in Sec.~\ref{ssec:useAnalysis}.
\end{example}

\subsection{Obstacle-induced SPP interaction}
\label{ssec:micromarcokernel}

In this section we show how properties of the interactions between SPPs and obstacles on the micro level inform the properties on the macro level and find some interesting connections to equations for granular flow, porous media and aggregation equations. We focus on the simplest case, where we assume the obstacle noise to be zero and consider the obstacle equation only until order $\gamma$. Further we work in one space dimension where many calculations can be done explicitly. Then the system of interest for the SPP density $\rho_g(x,t)$ and the obstacle density $\rho_f(x,t)$ for $x\in \mathbb{R}$ and $t\geq 0$ is given by \eqref{eqn:macroSw_rho_1D} coupled to \eqref{eqn:obsDensSimple_1D}.

\paragraph{A non-local equation with gradient flow structure.} If we substitute $\rho_f$ given in \eqref{eqn:obsDensSimple_1D} into the equation for $\rho_g$ given in \eqref{eqn:macroSw_rho_1D} we obtain
\begin{align}
\label{eqn:macro_gradientFlow}
& \dt \rho_g+c_1\pd_x \rho_g = \frac{1}{\zeta} \pd_x \left[ \rho_g \pd_x \left(\mu \rho_g+ \frac{\gamma}{\eta} \phi' *\phi'* \rho_g  \right)\right],
\end{align}
 We now see that we have a non-linear, non-local model with a gradient flow structure. These types of equations appear in a wide range of contexts ranging from granular flow, porous media and biological aggregation \cite{Otto2001, Topaz2006, Toscani2000} and their properties are subject of intense study \cite{Ambrosio2008, Benedetto1998, Carrillo2003}. The term stemming from the SPP self-repulsion is often written as $\mu \rho=H'(\rho)$, where $H(\rho)=\frac{\mu}{2}\rho^2$ is the SPP density of internal energy of the system. For the term stemming from the SPP-obstacle interaction, we can define the interaction kernel
\begin{align}
\label{eqn:W}
W(x)=( \phi' *\phi')(x).
\end{align}
Note that while $\phi$ is the microscopic interaction potential between SPPs and obstacles, $W$ can be interpreted as macroscopic obstacle-induced SPP interaction potential. 

\begin{figure}[t]
\centering
\includegraphics[width=\textwidth]{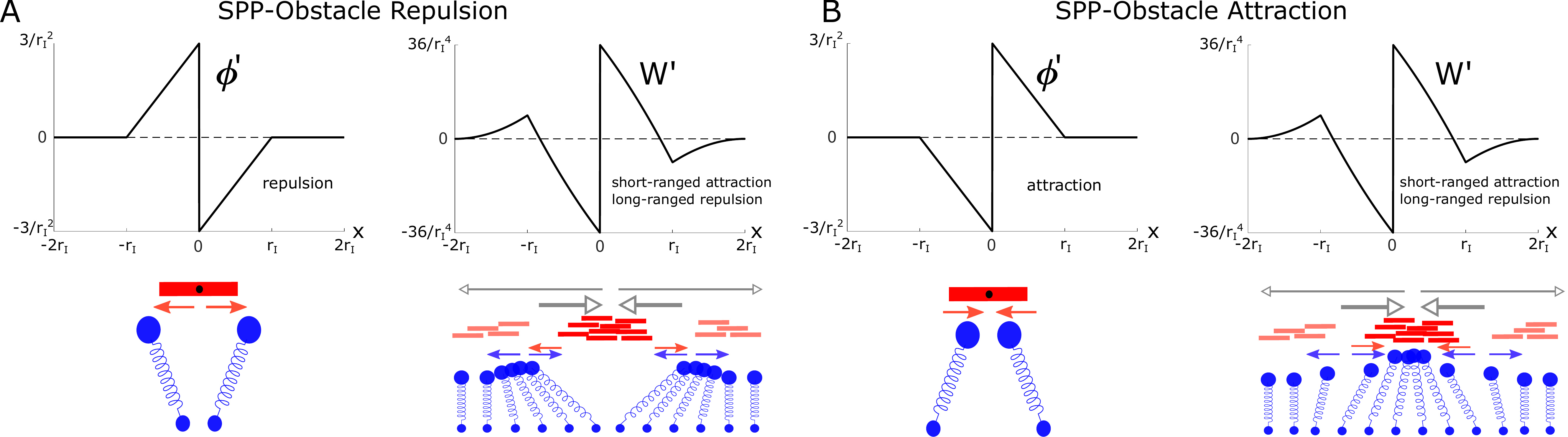}
\caption{\small{Micro-macro interactions. Shown are the microscopic SPP-obstacle interaction force $\phi'(x)$ and the resulting macroscopic interaction force $W'(x)$ for a 1D example, where $\phi(x)$ defined as in Ex.~\ref{ex:linStab}, for (A) $C=1$, i.e. repulsive and (B) $C=-1$, i.e. attractive. The schematic below illustrates the underlying interactions, where red and blue arrows mark the effect of the SPPs on the obstacles and vice versa respectively. Grey arrows show the net effect the group of SPPs in the center has on other SPPs.}}
\label{fig:SwObinterpret}
\end{figure}

\noindent\paragraph{Bi-phasic effect at the SPP level.}
We now infer properties of $W$ (the macro interaction potential) from properties of $\phi$ (the micro interaction potential). Note that $\phi'>0$ or $W'>0$ indicate forces to the left and $\phi'<0$ or $W'<0$ indicate forces to the right.

\begin{lemma}[Obstacle-induced SPP interactions]
\label{lem:micromacrokernel} Let $\phi(x)$ be an even potential. Then $\phi'$ is odd and we can define a function $\p$ on $[0,\infty)$ by
\begin{align}
\label{eqn:a}
\phi'(x)=\p(|x|)\,\sign(x),
\end{align}
using the convention that $\sign(0)=0$ and defining $\p(0):=\lim_{x\rightarrow 0^+}\p(x)$. Let $\p(0)\neq 0$ and $\p$ be continuous with bounded first derivative on $[0,\infty)$. We further assume that $\p$ has compact support on $[0,  r_I]$ for some $ r_I>0$, and that $\p$ and $\p'$ have constant but opposite sign on their support. Let $W$ be defined as in \eqref{eqn:W}. Then the following holds:
\begin{enumerate}
\item[(i)] $W$ is an even potential continuous on $\mathbb{R}$ and continuously differentiable on $\mathbb{R}\backslash\{0\}$. $W$ has compact support on $[-2 r_I, 2 r_I]$.
\item[(ii)] $W$ is an attractive potential for short distances, i.e. $W'(x)>0$ for $x>0$, $x$ small.
\item[(iii)] $W$ is a repulsive potential on $( r_I,2 r_I)$, i.e. $W'(x)<0$ for $x\in ( r_I,2 r_I)$.
\end{enumerate}
\end{lemma}
\begin{proof}
(i): Since $W$ is the convolution of two compactly supported, bounded functions, $W$ is continuous. Using the definition of $W$ and that $\phi'$ is odd, we calculate
\begin{align*}
W(-x)=&\int \phi'(y)\phi'(-x-y)\dd y\\
=&\int \phi'(-y)\phi'(-x+y)\dd y = \int \phi'(y)\phi'(x-y)\dd y=W(x).
\end{align*}
Using \eqref{eqn:a} we calculate $\phi''(x)=\p'(|x|)+2\p(0)\delta(x)$, where $\delta$ is the Dirac delta. We therefore obtain
\begin{align}
\label{eqn:dW}
W'(x)=&\,( \phi'' *\phi')(x)\nonumber\\
=&2\,\p(|x|)\p(0)\sign(x)+\int \p'(|z|)\p(|x-z|)\sign(x-z)\dd z.
\end{align}
The second term is continuous in $x$, since it is the convolution of two compactly supported functions, both bounded, in particular it is zero if evaluated at $x=0$ due to symmetry. The first term is continuous on $\mathbb{R}\backslash\{0\}$ hence the same is true for $W'$. That $W$ is compactly supported on $[-2 r_I,2 r_I]$ is a consequence of the support of $\phi'$.\newline
(ii): Using \eqref{eqn:dW} we find that
$$
\lim_{x\rightarrow 0^+}W'(x)=2 (\p(0))^2>0,
$$
which together with the results of (i), shows that $W'(x)>0$ for small, but positive $x$. This shows that $W$ is an attractive potential for small distances.\newline
(iii): Let $x\in( r_I, 2 r_I)$. Using \eqref{eqn:dW} we find that
\begin{align}
W'(x)=\int_{x- r_I}^{ r_I} \p'(z)\p(x-z)\dd z.
\end{align}
By assumption, the product of $\p'$ and $\p$ is negative, which shows that $W$ is an repulsive potential at distances between $ r_I$ and $2 r_I$. This finishes the proof.
\end{proof}

\begin{example}[Micro-macro potentials]
\label{ex:micromacrokernel} We illustrate the results of the above Lemma with two examples of SPP-obstacle potentials. Using the notation introduced in \eqref{eqn:a} we consider for $r\in[0, \infty)$
\begin{align*}
\p_1(r)=C \frac{3}{ r_I^2}\left(1-\frac{r}{ r_I}\right)H( r_I-r),\quad \p_2(r)=C \frac{1}{2r_I^2} e^{-r/r_I},
\end{align*}
where $H$ is the Heaviside function, $ r_I>0$. $\pm C>0$ corresponding to attractive and repulsive SPP-obstacle interactions respectively. The function $\p_1$ corresponds to the potential of Ex.~\ref{ex:linStab}, which is compactly supported and covered by Lem.~\ref{lem:micromacrokernel}, while $\p_2$ corresponds to a kernel without compact support. Fig.~\ref{fig:SwObinterpret}A and B shows the resulting obstacle-induced SPP forces $W_1'$ for $\p_1$ for $C=-1$ and $C=1$ respectively. For $\p_2$ we can see the bi-phasic behaviour directly by calculating
$$
W_2'(x)=\frac{1}{4r_I^5}e^{-\frac{|x|}{r_I}}\left(2r_I-|x|\right)\text{sign}(x),
$$
showing that $W_2$ is an attractive potential for $|x|<2r_I$ and repulsive otherwise. Note that for both examples the sign of $C$ doesn't affect the shape of $W'$.
\end{example}

Lem.~\ref{lem:micromacrokernel} and Ex.~\ref{ex:micromacrokernel} show that the SPP-obstacle interactions will have a short-ranged attractive effect on SPP level, irrespective of whether the micro interaction was attractive or repulsive. This can be understood intuitively, see Fig.~\ref{fig:SwObinterpret}: If the SPPs and obstacles repel each other, the obstacles that have been repelled by a group of SPPs, will in turn repel other SPPs and therefore lead to further aggregation of the SPPs (Fig.~\ref{fig:SwObinterpret}A). On the other hand if the SPPs attract the obstacles, the obstacles attracted by a group of SPPs will attract even more SPPs, again leading to an aggregation effect on the SPP level (Fig.~\ref{fig:SwObinterpret}B). Further Lem.~\ref{lem:micromacrokernel} shows that if the SPP-obstacle interaction force (whether attractive or repulsive) is falling with distance, we see that in addition to the short-ranged attraction, we have a long-ranged repulsion at the SPP level as well. The second example in Ex.~\ref{ex:micromacrokernel} suggests that this property is not limited to compactly supported functions and that Lem.~\ref{lem:micromacrokernel} can be generalized to a bigger class to kernels.\newline

These observations already give a good intuition to understand the phenomena observed in Sec.~\ref{ssec:IBMsim}. For both the moving clusters and the travelling bands, the 1D equations (along the global alignment direction) would correspond to moving aggregates of SPPs. Both the moving clusters and the travelling bands seems to have controlled size, in particular we observed that a clusters that is too big is split into two. The above observations now give an explanation for the observed behaviour: The SPP-obstacle interaction leads to short-ranged SPP attraction and hence aggregation, however, due to the two sources of repulsion (SPP self-repulsion and obstacle-induced repulsion), clusters cannot grow too large. Next we perform 1D simulations to compare the macro model with the IBM.


\section{Numerical results in 1D}
\label{sec:numerics}

In this section we numerically solve the 1D macro model for SPP-obstacle interactions and compare the results to both 1D IBM simulations and the analytical results of Sec.~\ref{sec:1dAnalysis}. Simulation details can be found in App.~\ref{ssec:simDetails1D}.

\subsection{Comparing SOH and IBM simulations}
\label{ssec:IBMvsSOH}

\begin{figure}[p]
\centering
\includegraphics[width=\textwidth]{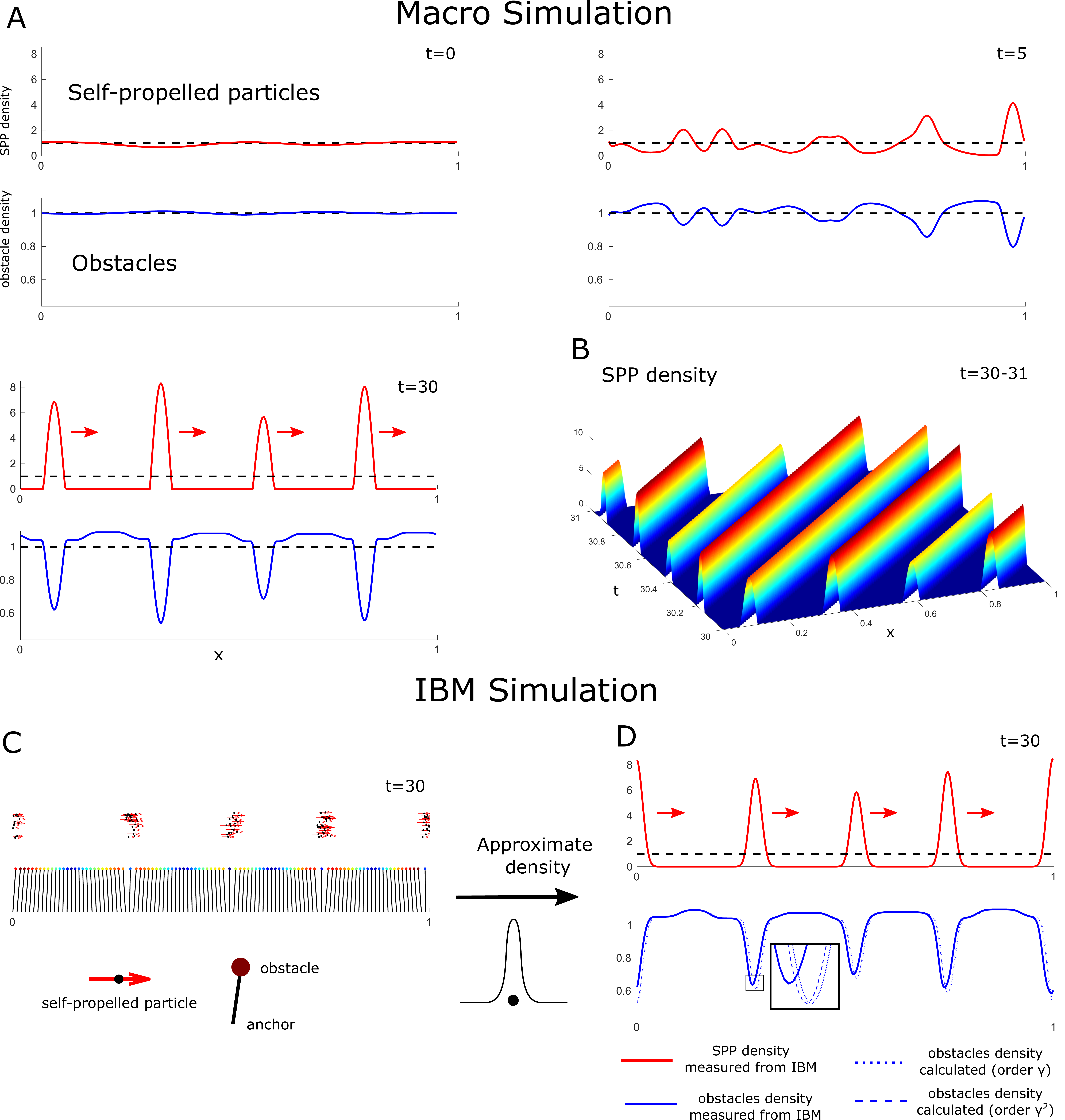}
\caption{\small{Simulations of the 1D macro model and IBM. A,D: Depicted are snapshots of numerical solutions showing the SPP (red, upper rows) and obstacle (blue, lower rows) densities, as well as their (constant) means (dashed black). A,B: Simulations of the macro model \eqref{eqn:macroSw_rho_1D}, \eqref{eqn:obsDensSimple_1D}. B: Space-time plot during one time unit of the continuation of the simulation in A. C,D: Simulations of the 1D IBM. C: Particle $x$-positions of the SPPs (red arrows with black dots) and obstacles (coloured circles, colour indicated displacement), $y$-positions are arbitrary. D: Approximated and calculated IBM particle densities of the simulation in C, see text for details.}}
\label{fig:sims1D}
\end{figure}

\paragraph{The macro SPP obstacle model produces travelling bumps.} We simulate \eqref{eqn:macroSw_rho_1D} coupled to \eqref{eqn:obsDensSimple_1D} in 1D using periodic boundary conditions on $x\in[0,1]$ and the following parameter choices: $\eta=1$, $c_1=1$, $\zeta=8$, $\gamma=2\times 10^{-3}$ and $\mu=5 \times 10^{-4}$. We use a linear microscopic interaction force, i.e. the interaction kernel as defined Ex.~\ref{ex:linStab} with $ r_I=0.18$ and $C=0.25$. As initial conditions we use a perturbed uniform SPP density. Fig.~\ref{fig:sims1D}A shows that, indeed, moving clusters of SPPs develop, with stretches of zero density between them. The clusters seem to be relatively evenly spread. The corresponding obstacle density is minimal where the SPP density is maximal. After the clusters have been established, we inspect the space-time plot for one time unit Fig.~\ref{fig:sims1D}B, which shows that they appear to be stably moving travelling waves of about speed one.

\paragraph{The macro SPP obstacle model agrees with the IBM.} Next we compare to 1D IBM simulations of \eqref{eqn:IBMmodel_massless}. Note that in 1D we can disregard the orientation equation and assume all particles self-propel to the right. We use the same parameters as for the macro model with $N=M=100$ and a self-repulsion kernel yielding a linear force, dropping with distance of width $r_R=0.02$. As initial conditions we use equally spaced anchor points and randomly positioned SPPs. Fig.~\ref{fig:sims1D}C shows the obstacles, their tether points and the SPPs at time $t=30$. We calculate the corresponding SPP and obstacle densities from the particle positions. To that end we create a smoothed version of the empirical distribution defined analogous to \eqref{eqn:empDist}, where the Dirac delta distributions have been replaced by 1D-Gaussians with variance $1\times 10^{-4}$. Note that choice of the variance is delicate, since it has to be small enough to be able to resolve the patterns and big enough to lead to meaningful averaging. The result is shown Fig.~\ref{fig:sims1D}D. A comparison between the simulated SPP and obstacle densities for the macro model and IBM shows remarkable good agreement both qualitatively and quantitatively.

\paragraph{Higher order approximations lead to a delay effect.} The macro model was simulated using an order $\gamma$ approximation for the obstacles. To assess the effect of the order $\gamma^2$ terms without solving the full system, we proceed as follows: We substitute the measured IBM SPP density depicted in Fig.~\ref{fig:sims1D}D into \eqref{eqn:obsDensMax_1D} (with $\delta =0$) to calculate the obstacle density as predicted by the model. We calculate both the order $\gamma$ and order $\gamma^2$ approximations. For the latter we need the time derivative of the SPP density, which we approximate by calculating the SPP density at the previous time step and using a forward finite difference approximation. The resulting densities are shown in Fig.~\ref{fig:sims1D}D. We observe that measured and calculated obstacle densities agree remarkably well. Inspecting the inset in Fig.~\ref{fig:sims1D}D, we see that the order $\gamma$ approximation predicts the obstacle density minima to be precisely at the SPP density maxima, however both the order $\gamma^2$ approximation and the actual measured IBM obstacle density have their local minima shifted backwards with respect to the SPP direction, yielding a better fit between the measured and calculated order $\gamma^2$ densities that those of the order $\gamma$ approximation. This demonstrates that the derived obstacle equation allows to calculate the obstacle density for a given SPP density. It also shows that the higher order approximation in $\gamma$ is necessary if one wants to account for effects of SPP movement.

\begin{figure}[t]
\centering
\includegraphics[width=0.75\textwidth]{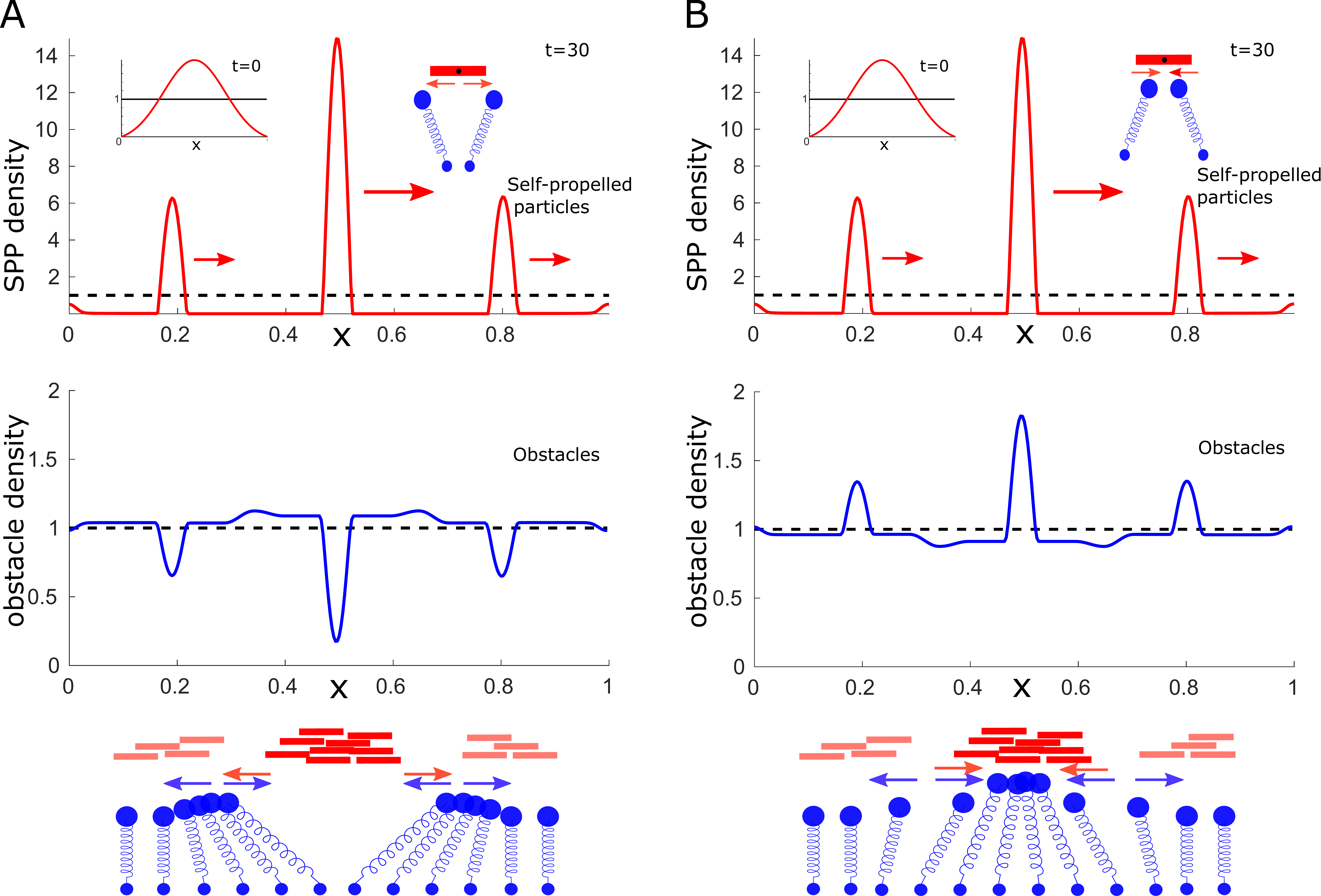}
\caption{\small{Simulations of the 1D macro model. A,B: Depicted are snapshots of numerical solutions to \eqref{eqn:macroSw_rho_1D}, \eqref{eqn:obsDensSimple_1D} showing the SPP (red, upper rows) and obstacle (blue, lower rows) densities, as well as their (constant) means (dashed black). Insets show initial condition, schematic depicts nature of microscopic interaction, red arrows indicate movement direction of the densities. Schematics in lower row: see Fig.~\ref{fig:SwObinterpret}.}}
\label{fig:sims1D_att_rep}
\end{figure}

\subsection{Testing analytical insights}
\label{ssec:useAnalysis}

\paragraph{Attractive and repulsive interactions lead to the same SPP behaviour.} In the next numerical experiment, shown in Fig.~\ref{fig:sims1D_att_rep} we use as initial condition a centrally placed Gaussian and inspect the moving steady state density for a repulsive (A) and an attractive (B) microscopic interaction force. We see that in both cases the resulting SPP density is the same, forming a travelling wave with a stable shape. This shape consists of a large cluster and two smaller clusters to its left and right. However, the obstacle density differs in the two cases: For an attractive potential we have obstacles clusters coinciding with the SPP clusters, whilst for the repulsive potential SPP clusters create regions of low obstacle density. The lower row compares this with the intuitive explanation of the previous section (see Fig.~\ref{fig:SwObinterpret}).

\paragraph{Linear stability analysis predicts macro and IBM patterns.} In Sec.~\ref{ssec:stabAnalysis} we performed a linear stability analysis for the 1D macro equation. In Ex.~\ref{ex:linStab} we determined the criteria for pattern formation and how to predict pattern size for a specific interaction potential shape. Now we compare these predictions to simulations of both the 1D macro equations \eqref{eqn:macroSw_rho_1D}, \eqref{eqn:obsDensSimple_1D} and the 1D IBM simulations by varying the size of the support of the interaction kernel $r_I$. We use the same kernels and number of particles as above and the following parameters: $\eta=1$, $c_1=1$, $\zeta=8$, $\gamma=2\times 10^{-3}$ and $\mu=6.7 \times 10^{-3}$, $C=0.17$. We start with a randomly perturbed constant initial density for the macro model and regularly spaced anchors and randomly placed SPPs for the IBM. We compare the predicted number of peaks as calculated in Ex.~\ref{ex:linStab} (and defined as the reciprocal of the pattern size) to the observed number of peaks at time $t=30$. The result is shown in Fig.~\ref{fig:stabAna_SOH_IBM}. We find that the analytical predictions of Sec.~\ref{ssec:stabAnalysis} agree very well with the macro model. The agreement with the IBM simulations is good as long as the macro model gives physically meaningful (i.e. positive) obstacle densities (examples 1,2,3 in Fig.~\ref{fig:stabAna_SOH_IBM}), but breaks down otherwise (example 4 in Fig.~\ref{fig:stabAna_SOH_IBM}). This shows both that the macro model can be used to gain insights into the IBM, but also that it is limited to certain parameter regimes.

\begin{figure}[t]
\centering
\includegraphics[width=\textwidth]{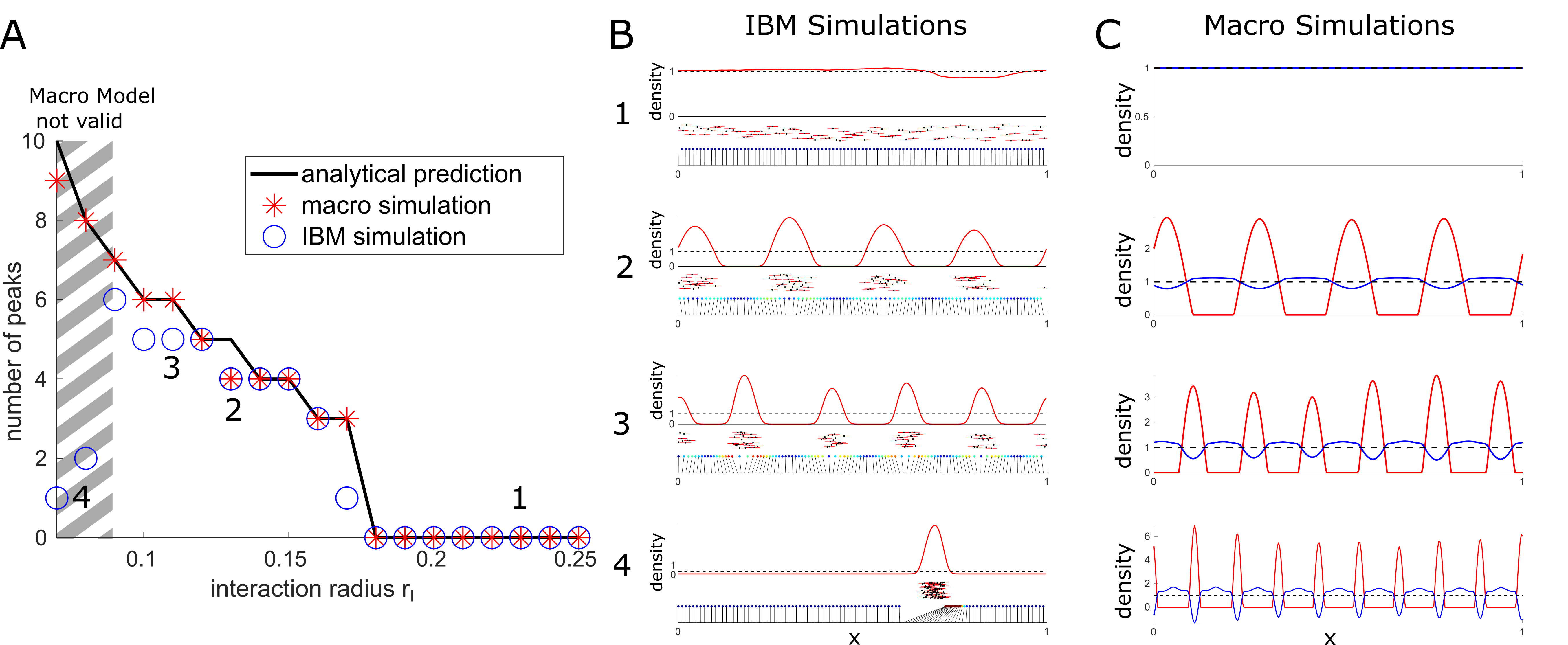}
\caption{\small{Linear stability analysis predictions. A: Number of peaks in dependence of interaction radius $ r_I$. Shown are the analytical predictions (black line), the macro simulation results (red star) and the IBM simulation results (blue circle). Numbers mark the examples in B and C. B,C: Final simulation results for the examples marked in A for the IBM (B) and the macro model (C). IBM results are depicted as in Fig.~\ref{fig:sims1D}, macroscopic SPP and obstacle densities are shown in red and blue respectively.}}
\label{fig:stabAna_SOH_IBM}
\end{figure}

\section{Discussion}

In this work we formulated an IBM model of the interaction of self-propelled, collectively moving SPPs with elastically tethered obstacles. Despite the seemingly simplicity of the interactions, we found that the system can self-organize into a big variety of patterns, including travelling bands, (transiently stable) trails and size controlled clusters. To investigate these patterns further we derived macroscopic equations for the obstacle and SPP densities and the SPP orientation. The asymptotic regime of interest assumed $\gamma$ to be small, i.e. fast obstacle spring relaxation (strong obstacle springs). The resulting continuum equations are non-linear and contain a non-local interaction term. Linear stability analysis revealed that the SPP-obstacle interactions have to be strong enough compared to the SPP self-repulsion to allow for patterns to evolve and allowed to estimate pattern size as a function of model parameters. We found that, surprisingly, SPP dynamics are independent of whether obstacles and SPPs repel or attract each other.\newline

In Sec.~\ref{sec:1dAnalysis} we discovered that the macroscopic SPP equation has gradient flow structure with a bi-phasic (short-range attractive, long-range repulsive) non-local obstacle-induced interaction kernel. Strong analytical results, such as energy dissipation estimates, exist for these type of equations, which suggests that it is possible, at least for certain cases, to construct steady states and assess their stability in a rigorous manner. Obvious extensions include 1D simulations of the SPP-obstacle model using an $\mathcal{O}(\gamma^2)$ approximation of the obstacle density or including the positional noise, as well as performing 2D or 3D simulations with the continuum model for various orders of approximations and systematic comparison with the IBM model.\newline

We found that both attractive or repulsive microscopic interactions between SPPs and obstacles cause a short-range attractive macroscopic effect on the SPP level, which leads to clustering. Clustering of organisms is ubiquitous in nature and is often attributed to direct attraction between the individuals. However our results suggest that the apparent attraction could be indirect and is in fact mediated by the environment. In other words it is possible the individuals feel no attraction towards each other, but will still form tight clusters. This could be relevant for example to understanding cell clustering or swarm formation.\newline

Our derivation relied heavily on the assumptions of smallness of $\gamma$. Mathematically this limitation manifests in the fact that the obstacle density can become negative, at which point the model becomes invalid. In the future we would like to derive macroscopic models that are and remain well-posed for any parameter combination. This will require a different closure method of the kinetic equations. Our current model seems to be able to capture several of the observed phenomena at the IBM model, such as the travelling bands or the clusters, however, for example the trail formation pattern will most likely require an extension of the current techniques.\newline

The current model describes interactions between SPPs and obstacles. In many instances, however, all components are immersed in a fluid. Past work has already studied how to derive and analyse SPP-fluid interactions \cite{Degond2019}. There exist models for how fluid properties are affected if it contains immersed objects. A famous example is the Oldroyd-B model, describing the visco-elasticity of fluids filled with spring dumbbells \cite{Oldroyd1950}. We plan to use our derivation strategy to derive equations for fluids filled with tethered obstacles and study how fluid properties such as viscosity are affected. An additional level of complexity we plan to tackle, is to combine all three components, the fluid, the obstacles and the SPPs. In this case a natural question appears: How big are the obstacles compared to the SPPs. The flexible techniques developed in this work will allow to answer this question by performing the coarse-graining at different levels.

\section*{Supplementary Material}

\paragraph{IBM Simulation Videos.}

The three supplementary videos
\begin{itemize}
\item \href{https://angelikamanhart.github.io/files/Movies/moving-clusters_smaller.avi}{moving-clusters.avi}
\item \href{https://angelikamanhart.github.io/files/Movies/trails_smaller.avi}{trails.avi}
\item \href{https://angelikamanhart.github.io/files/Movies/travelling-bands_smaller.avi}{travelling-bands.avi}
\end{itemize}
show the dynamics in time of the 2D IBM simulations depicted in Fig.~\ref{fig:IBM_patterns}. SPPs are shown in red, obstacles in blue.

\appendix

\section{Appendix}

\subsection{Simulation details}
\label{ssec:simDetails1D}

{ \it IBM simulations of Sec.~\ref{ssec:IBMsim}:}  We simulate the IBM model \eqref{eqn:IBMmodel_massless} in two space dimensions using Matlab with a timestep of $\Delta t=10^{-3}$. Model parameters are listed in Sec.~\ref{ssec:IBMsim}. Numerically we use the circle method described in \cite{Motsch_Navoret_MMS11}.\newline

\noindent{\it Macro-model simulations in 1D of Sec.~\ref{sec:numerics}:} We simulation the macroscopic model \eqref{eqn:macroSw_rho_1D}, \eqref{eqn:obsDensSimple_1D} in one space dimension using Matlab with spatial and temporal timesteps of $\Delta x=3\times 10^{-3}$, $\Delta t=10^{-2}$. The method used is described in \cite{Carrillo2015}.

\subsection{Properties of the operator $\mathcal{B}$ defined in \eqref{eqn:FokkerPlanck}}
\label{ssec:FPO}

It is a well-known fact that the operator $\mathcal{B}$ defined in \eqref{eqn:FokkerPlanck} is the generator of the 
Ornstein-Uhlenbeck stochastic process (see \cite{pavliotis2014}). 
For an extensive study of this operator we refer the reader to \cite{Achleitner2015}, \cite{Risken1996}, or \cite{RS1978}.
In the following result we collect a few properties needed in this paper.

\begin{lemma}[Properties of $\mathcal{B}$]
\label{lem:FokkerPlanck}
Let the operator $\mathcal{B}$ be defined by \eqref{eqn:FokkerPlanck} and let $\bm{i}=(i_1,i_2,i_3)$ be a multi-index. 
We define
\begin{align*}
\H_{\bm{i}}(\sigma)=H_{i_1}(\sigma_1)H_{i_2}(\sigma_2)H_{i_3}(\sigma_3),
\end{align*}
where
\begin{align*}
H_j(s)= (-1)^j e^{\frac{s^2}{2}}\frac{\dd^j}{\dd s^j} e^{-\frac{s^2}{2}}.
\end{align*}
Note that $H_j(s)$ are the (probabilistic) Hermite polynomials. Let us consider the following $L^2$-weighted space
\[
X := \bigg\{ f \in L^2( \R^3) : \int_{\R^3} f^2 M_1 \dd \sigma < \infty \bigg\} \, , 
\]
where $M_1$ is defined in \eqref{eqn:Gaussian} taking $\delta = 1$. For any two functions $f$ and $g$ in $X$, we define their weighted inner product by
\begin{align*}
\langle f, g \rangle_X := \int_{\R^3} f(\sigma) g(\sigma)M_1(\sigma) \dd \sigma.
\end{align*}
We then have the following properties:
\begin{enumerate}
\item[P1.] $\langle \H_{\bm{i}}, \H_{\bm{j}} \rangle_X = \bm{i}! \delta_{\bm{i}\bm{j}}$, where $\delta_{\bm{i}\bm{j}}$ is the Kronecker delta for multi-indices.
\item[P2.] $\mathcal{B}(\H_{\bm{i}})=-|\bm{i}|\H_{\bm{i}}$.
\item[P3.] The set $\{\H_{\bm{i}}\}_{\bm{i}}$ is a complete orthogonal basis of the $L^2$-weighted space $X$.
\item[P4.] $\H_{e_i}\H_{e_k}=\H_{e_i+e_k}+\delta_{ik}\H_{\bm{0}}$.
\item[P5.] $\H_{e_i}\H_{e_j + e_k}=\H_{e_i+e_j+e_k}+\delta_{ik}\H_{e_j}+\delta_{ij}\H_{e_k}$.
\end{enumerate}
We have used the notation $\bm{i}!=i_1!i_2!i_3!$ and $|\bm{i}| = i_1+i_2+i_3$. 
Note that P1 shows that $\H_{\bm{i}}$ are orthogonal with respect to the inner product $\langle \cdot, \cdot \rangle_X$
and P2 states that $\H_{\bm{i}}$ are eigenfunctions with eigenvalue $-|\bm{i}|$. 
In the product rules P4 and P5, $e_i$ denotes the $i$-th unit vector in $\R^3$.
\end{lemma}


\subsection{Calculation of the obstacle density.}
\label{ssec:moments}

In this section we detail the calculations of 0-$th$ order moment of $f$, $\rho_f$, in terms of expansions with respect to $\gamma$ and $\delta $. As outlined in the main text we will perform the following steps:
\begin{enumerate}
\item Perform the change of variables $\sqrt{\delta }\sigma=x-y$. This changes the integrand to be proportional to $M_1(\sigma)h_i(\sigma, x-\sqrt{\delta }\sigma,t)$ for $\rho_{f_i}$.
\item Next we Taylor expand $h_i(\sigma,x-\sqrt{\delta }\sigma,t)$ around $\delta =0$ using the expansions of above.
\item Then we calculate the contributions using the scaling condition \eqref{eqn:scaling2}, 
the orthogonality of the Hermite polynomials (P1) and the product rule (P4) of Lemma \ref{lem:FokkerPlanck}.
\end{enumerate}

The following result will be helpful for the subsequent calculations.
\begin{lemma}
\label{lem:evenOdd}
Let $h_1^k$ and $h_2^k$, for $k = 0, 1, \ldots$, be the solutions of the above expansion and let their representations w.r.t the basis of Hermite polynomials be given by
\begin{align*}
h_1^k=\sum_{\bm{i}} a^k_{\bm{i}} \H_{\bm{i}}(\sigma), \qquad h_2^k=\sum_{\bm{i}} b^k_{\bm{i}} \H_{\bm{i}}(\sigma),
\end{align*}
where $\bm{i}$ is a multiindex and $a^k_{\bm{i}}$ and $b^k_{\bm{i}}$ are functions of $y$ and $t$. Then it holds that
\begin{align*}
a^k_{\bm{i}}\equiv 0 \qquad \text{for} \quad |\bm{i}|\!\!\! \mod 2 = k \!\!\!\mod 2, \quad \text{or} \quad |\bm{i}|>k+1\\
b^k_{\bm{i}}\equiv 0 \qquad \text{for} \quad |\bm{i}| \!\!\!\mod 2 \neq k \!\!\!\mod 2, \quad \text{or} \quad |\bm{i}|>k+2
\end{align*}
\end{lemma}
\begin{proof}
This can be shown by induction. For the initial case we use the explicit solutions given in \eqref{eqn:h1sol} and \eqref{eqn:h2sol}.
\end{proof}

\noindent\textbf{Notation:} In the following, if no further argument is given, functions are evaluated at $(\sigma,x,t)$ and $\pd_i:=\frac{\pd}{\pd x_i}$. We use the Einstein summation convention. In general we often suppress the dependence on time $t$.\newline

\begin{remark}
\label{rmk:evenOdd}
In the following we often use Lemma \ref{lem:evenOdd} in combination with the fact that odd-order moments of $M_1$ are zero.
\end{remark}

\noindent Preparation for Step 2 in the above procedure: Taylor expand $h_r(\sigma,x-\sqrt{\delta }\sigma,t)$, 
\begin{align}
\label{eqn:taylor_h}
h_r(\sigma, x-\sqrt{\delta }\sigma)=&h_r^0-\sqrt{\delta }\sigma_k \pd_k h_r^0+\sqrt{\delta }h_r^1\\
&+\frac{\delta }{2}\sigma_i\sigma_j\pd_{ij}h_r^0-\delta  \sigma_i \pd_i h_r^1 + \delta  h_r^2 + \mathcal{O}(\delta ^{3/2}),\nonumber
\end{align}
where $r = 1, 2$.
We start with the \textbf{0-th order density $\rho_{f_0}$}:
\begin{align*}
\rho_{f_0}(x,t)=\int f_0(x,y)\dd y= \int M_\delta (x-y) \dd y=1
\end{align*}
For the \textbf{first order density $\rho_{f_1}$} we use the reformulation in terms of $h_1$:
\begin{align*}
\rho_{f_1}(x,t)&=\int f_1(x,y)\dd y=\frac{1}{\sqrt{\delta }} \int M_\delta (x-y) h_1\left(\frac{x-y}{\sqrt{\delta }},y\right) \dd y\\
&=\frac{1}{\sqrt{\delta }} \int M_1(\sigma) h_1\left(\sigma,x-\sqrt{\delta }\sigma\right) \dd \sigma\\
&=\langle 1, -\sigma_k \pd_k h_1^0\rangle + \mathcal{O}(\delta )
\end{align*}
In the second line we have used Step 1, the change of variables. In the third line we have used \eqref{eqn:taylor_h} for $r=1$ together with the fact that the order $\delta ^{-1/2}$ term and the order 1 term involving $h_1^1$ are zero due to the normalization condition \eqref{eqn:scaling2}. For the order $\sqrt{\delta }$-terms we used Rmk.~\ref{rmk:evenOdd} to show it is zero.\newline 
Hence we are left with one term. We use \eqref{eqn:h1sol} and calculate
\begin{align*}
\langle 1, -\sigma_k \pd_k h_1^0\rangle &=- \pd_k \ve_i\langle \H_{e_k}, \H_{e_i}\rangle=- \pd_k \ve_k,
\end{align*}
where we have used P1 of Lemma \ref{lem:FokkerPlanck}, i.e. $\langle \H_{e_k},\H_{e_i}\rangle=\delta_{ki}$. This shows that indeed
\begin{align*}
\rho_{f_1}(x,t)&=- \pd_k \ve_k+ \mathcal{O}(\delta )
\end{align*}
and finishes the calculations for $\rho_{f_1}$. The calculations for the order $\delta $ term are similar and omitted here.\newline

\noindent We continue in similar fashion with the \textbf{second order density $\rho_{f_2}$}: We use the reformulation in terms of $h_2$ and get

\begin{align*}
\rho_{f_2}(x,t)&=\int f_2(x,y)\dd y=\frac{1}{\delta } \int M_\delta (x-y) h_2\left(\frac{x-y}{\sqrt{\delta }},y\right) \dd y\\
&=\frac{1}{\delta } \int M_1(\sigma) h_2\left(\sigma,x-\sqrt{\delta }\sigma\right) \dd \sigma.
\end{align*}
If we now inspect \eqref{eqn:taylor_h} for $r=2$, we find that, as above the scaling condition \eqref{eqn:scaling2} 
leads to the $\delta ^{-1}$ order term involving $h_2^0$, the $\delta ^{-1/2}$-order term involving $h_2^1$ and the order 
one term involving $h_2^2$ being zero. For the remaining $\delta ^{-1/2}$-order term we refer to 
Rmk.~\ref{rmk:evenOdd} and hence it is also 0. For the remaining order one terms we calculate
\begin{align*}
A_1&:=\frac{1}{2} \int M_1(\sigma) \sigma_i\sigma_j \pd_{ij}h_2^0 \dd \sigma =\frac{1}{4} \pd_{ij}(\ve_k\ve_l)\int M_1(\sigma) \sigma_i\sigma_j \H_{e_k+e_l} \dd \sigma\\
&=\frac{1}{4} \pd_{ij}(\ve_k\ve_l)\langle \H_{e_i+e_j}, \H_{e_k+e_l}\rangle \, ,\\
A_2&:=- \int M_1(\sigma) \sigma_i \pd_{i}h_2^1 \dd \sigma =\pd_i(\pd_t \ve_k-\ve_j \pd_j \ve_k) \int M_1(\sigma) \sigma_i \H_{e_k}\dd \sigma\\
&=\pd_i(\pd_t \ve_k-\ve_j \pd_j \ve_k) \langle \H_{e_i}, \H_{e_k}\rangle \, ,
\end{align*}
where we have used P1 and P4 of Lem.~\ref{lem:FokkerPlanck}. We continue
\begin{align*}
A_1 & = \frac{1}{2}\pd_{ij}(\ve_i\ve_j)       \, ,\\
A_2 & = \pd_i(\pd_t \ve_i-\ve_j \pd_j \ve_i) \, ,
\end{align*}
where we used the identity
\[
\langle \H_{e_i+e_j}, \H_{e_k+e_l}\rangle=\delta_{ik}\delta_{jl}+\delta_{il}\delta_{jk}
\]
for $A_1$. Finally we calculate
\begin{align*}
\rho_{f_2}(x,t)&=A_1+A_2+\mathcal{O}(\delta )=\left\{\pd_t \pd_i \ve_i+\frac{1}{2}\pd_i\left[\ve_i \pd_j \ve_j-\ve_j\pd_j \ve_i\right] \right\}+\mathcal{O}(\delta ).
\end{align*}
Note the the fact that the remaining term is $\mathcal{O}(\delta )$ and not $\mathcal{O}(\sqrt{\delta })$ is again thanks to Rmk.~\ref{rmk:evenOdd} 
and does not require explicit knowledge of the shape of $h_2^3$.

\bibliographystyle{abbrv}
\bibliography{sperm_bibliography}

\begin{thebibliography}{10}

\bibitem{AS2019}
P.~Aceves-Sanchez, M.~Bostan, J.~A. Carrillo, and P.~Degond.
\newblock Hydrodynamic limits for kinetic flocking models of cucker-smale type.
\newblock {\em Mathematical Biosciences and Engineering}, 16(6):7883--7910,
  2019.

\bibitem{Achleitner2015}
F.~Achleitner, A.~Arnold, and D.~St{\"u}rzer.
\newblock Large-time behavior in non-symmetric fokker-planck equations.
\newblock {\em Rivista di Matematica della Universit\'a di Parma}, 6:1--68,
  2015.

\bibitem{Ambrosio2008}
L.~Ambrosio, N.~Gigli, and G.~Savar{\'e}.
\newblock {\em Gradient flows: in metric spaces and in the space of probability
  measures}.
\newblock Springer Science \& Business Media, 2008.

\bibitem{Baricos1995}
W.~H. Baricos, S.~L. Cortez, S.~S. El-Dahr, and H.~W. Schnaper.
\newblock Ecm degradation by cultured human mesangial cells is mediated by a
  pa/plasmin/mmp-2 cascade.
\newblock {\em Kidney international}, 47(4):1039--1047, 1995.

\bibitem{Ben-Jacob2000}
E.~Ben-Jacob, I.~Cohen, and H.~Levine.
\newblock Cooperative self-organization of microorganisms.
\newblock {\em Adv. in Phys.}, 49(4):395--554, 2000.

\bibitem{Benedetto1998}
D.~Benedetto, E.~Caglioti, J.~A. Carrillo, and M.~Pulvirenti.
\newblock A non-maxwellian steady distribution for one-dimensional granular
  media.
\newblock {\em Journal of statistical physics}, 91(5-6):979--990, 1998.

\bibitem{Bernoff2016}
A.~J. Bernoff and C.~M. Topaz.
\newblock Biological aggregation driven by social and environmental factors: A
  nonlocal model and its degenerate cahn--hilliard approximation.
\newblock {\em SIAM Journal on Applied Dynamical Systems}, 15(3):1528--1562,
  2016.

\bibitem{Boissard2013}
E.~Boissard, P.~Degond, and S.~Motsch.
\newblock Trail formation based on directed pheromone deposition.
\newblock {\em Journal of mathematical biology}, 66(6):1267--1301, 2013.

\bibitem{Buhl2006}
J.~Buhl, D.~Sumpter, I.~Couzin, J.~Hale, E.~Despland, E.~Miller, and
  S.~Simpson.
\newblock From disorder to order in marching locusts.
\newblock {\em Science}, 312(5778):1402--1406, 2006.

\bibitem{Carrillo2015}
J.~A. Carrillo, A.~Chertock, and Y.~Huang.
\newblock A finite-volume method for nonlinear nonlocal equations with a
  gradient flow structure.
\newblock {\em Communications in Computational Physics}, 17(1):233--258, 2015.

\bibitem{Carrillo2003}
J.~A. Carrillo, R.~J. McCann, C.~Villani, et~al.
\newblock Kinetic equilibration rates for granular media and related equations:
  entropy dissipation and mass transportation estimates.
\newblock {\em Revista Matematica Iberoamericana}, 19(3):971--1018, 2003.

\bibitem{Cavagna2010}
A.~Cavagna, A.~Cimarelli, I.~Giardina, G.~Parisi, R.~Santagati, F.~Stefanini,
  and M.~Viale.
\newblock Scale-free correlations in starling flocks.
\newblock {\em Proc. Natl. Acad. Sci. USA}, 107(26):11865--11870, 2010.

\bibitem{chepizhko2013optimal}
O.~Chepizhko, E.~G. Altmann, and F.~Peruani.
\newblock Optimal noise maximizes collective motion in heterogeneous media.
\newblock {\em Physical review letters}, 110(23):238101, 2013.

\bibitem{chepizhko2013diffusion}
O.~Chepizhko and F.~Peruani.
\newblock Diffusion, subdiffusion, and trapping of active particles in
  heterogeneous media.
\newblock {\em Physical review letters}, 111(16):160604, 2013.

\bibitem{Cheung2016}
K.~J. Cheung and A.~J. Ewald.
\newblock A collective route to metastasis: Seeding by tumor cell clusters.
\newblock {\em Science}, 352(6282):167--169, 2016.

\bibitem{Creppy2016}
A.~Creppy, F.~Plourabou{\'e}, O.~Praud, X.~Druart, S.~Cazin, H.~Yu, and
  P.~Degond.
\newblock Symmetry-breaking phase transitions in highly concentrated semen.
\newblock {\em Journal of The Royal Society Interface}, 13(123):20160575, 2016.

\bibitem{DeGennes_Prost_1993}
P.~G. de~Gennes and J.~Prost.
\newblock {\em The physics of liquid crystals}.
\newblock Oxford University Press, 1993.

\bibitem{Degond2015}
P.~Degond, G.~Dimarco, T.~B.~N. Mac, and N.~Wang.
\newblock Macroscopic models of collective motion with repulsion.
\newblock {\em Communications in Mathematical Sciences}, 13:1615--1638, 2015.

\bibitem{Manhart2018}
P.~Degond, A.~Manhart, and H.~Yu.
\newblock An age-structured continuum model for myxobacteria.
\newblock {\em Mathematical Models and Methods in Applied Sciences},
  28(09):1737--1770, 2018.

\bibitem{Degond2019}
P.~Degond, S.~Merino-Aceituno, F.~Vergnet, and H.~Yu.
\newblock Coupled self-organized hydrodynamics and stokes models for
  suspensions of active particles.
\newblock {\em Journal of Mathematical Fluid Mechanics}, 21(1):6, 2019.

\bibitem{Degond2008}
P.~Degond and S.~Motsch.
\newblock Continuum limit of self-driven particles with orientation
  interaction.
\newblock {\em Math. Models Methods Appl. Sci.}, 18(Suppl.):1193--1215, 2008.

\bibitem{Feliciani2016}
C.~Feliciani and K.~Nishinari.
\newblock Empirical analysis of the lane formation process in bidirectional
  pedestrian flow.
\newblock {\em Physical Review E}, 94(3):032304, 2016.

\bibitem{Ha2008}
S.-Y. Ha and E.~Tadmor.
\newblock From particle to kinetic and hydrodynamic descriptions of flocking.
\newblock {\em Kinetic \& Related Models}, 1:415, 2008.

\bibitem{Helbing1997}
D.~Helbing, J.~Keltsch, and P.~Molnar.
\newblock Modelling the evolution of human trail systems.
\newblock {\em Nature}, 388(6637):47--50, 1997.

\bibitem{jabbarzadeh2014swimming}
M.~Jabbarzadeh, Y.~Hyon, and H.~C. Fu.
\newblock Swimming fluctuations of micro-organisms due to heterogeneous
  microstructure.
\newblock {\em Physical Review E}, 90(4):043021, 2014.

\bibitem{kamal2018enhanced}
A.~Kamal and E.~E. Keaveny.
\newblock Enhanced locomotion, effective diffusion and trapping of undulatory
  micro-swimmers in heterogeneous environments.
\newblock {\em Journal of the Royal Society Interface}, 15(148):20180592, 2018.

\bibitem{Lam1995}
L.~Lam.
\newblock Active walker models for complex systems.
\newblock {\em Chaos, Solitons \& Fractals}, 6:267--285, 1995.

\bibitem{Lo2000}
C.-M. Lo, H.-B. Wang, M.~Dembo, and Y.-l. Wang.
\newblock Cell movement is guided by the rigidity of the substrate.
\newblock {\em Biophysical journal}, 79(1):144--152, 2000.

\bibitem{majmudar2012experiments}
T.~Majmudar, E.~E. Keaveny, J.~Zhang, and M.~J. Shelley.
\newblock Experiments and theory of undulatory locomotion in a simple
  structured medium.
\newblock {\em Journal of the Royal Society Interface}, 9(73):1809--1823, 2012.

\bibitem{Mogilner2016}
A.~Mogilner and A.~Manhart.
\newblock Agent-based modeling: case study in cleavage furrow models.
\newblock {\em Molecular biology of the cell}, 27(22):3379--3384, 2016.

\bibitem{Motsch_Navoret_MMS11}
S.~Motsch and L.~Navoret.
\newblock Numerical simulations of a nonconvervative hyperbolic system with
  geometric constraints describing swarming behavior.
\newblock {\em Multiscale Model. Simul.}, 9:1253--1275, 2011.

\bibitem{Oldroyd1950}
J.~G. Oldroyd.
\newblock On the formulation of rheological equations of state.
\newblock {\em Proceedings of the Royal Society of London. Series A.
  Mathematical and Physical Sciences}, 200(1063):523--541, 1950.

\bibitem{Otto2001}
F.~Otto.
\newblock The geometry of dissipative evolution equations: the porous medium
  equation.
\newblock 2001.

\bibitem{park2008enhanced}
S.~Park, H.~Hwang, S.-W. Nam, F.~Martinez, R.~H. Austin, and W.~S. Ryu.
\newblock Enhanced caenorhabditis elegans locomotion in a structured
  microfluidic environment.
\newblock {\em PloS one}, 3(6), 2008.

\bibitem{pavliotis2014}
G.~A. Pavliotis.
\newblock {\em Stochastic processes and applications: diffusion processes, the
  Fokker-Planck and Langevin equations}, volume~60.
\newblock Springer, 2014.

\bibitem{Peurichard2016}
D.~Peurichard.
\newblock Macroscopic model for cross-linked fibers with alignment
  interactions: Existence theory and numerical simulations.
\newblock {\em Multiscale Modeling \& Simulation}, 14(4):1175--1210, 2016.

\bibitem{RS1978}
M.~Reed and B.~Simon.
\newblock {\em Methods of Modern Mathematical Physics. IV. Analysis of
  Operators. 1978}.
\newblock Academic Press, New York.

\bibitem{Risken1996}
H.~Risken.
\newblock Fokker-planck equation.
\newblock In {\em The Fokker-Planck Equation}, pages 63--95. Springer, 1996.

\bibitem{Schoeller2018}
S.~F. Schoeller and E.~E. Keaveny.
\newblock From flagellar undulations to collective motion: predicting the
  dynamics of sperm suspensions.
\newblock {\em Journal of The Royal Society Interface}, 15(140):20170834, 2018.

\bibitem{Shaw1978}
E.~Shaw.
\newblock Schooling fishes: the school, a truly egalitarian form of
  organization in which all members of the group are alike in influence, offers
  substantial benefits to its participants.
\newblock {\em American Scientist}, 66(2):166--175, 1978.

\bibitem{Shimkets1990}
L.~J. Shimkets.
\newblock Social and developmental biology of the myxobacteria.
\newblock {\em Microbiol. Rev.}, 54(4):473--501, 1990.

\bibitem{Sokolov2007}
A.~Sokolov, I.~S. Aranson, J.~O. Kessler, and R.~E. Goldstein.
\newblock Concentration dependence of the collective dynamics of swimming
  bacteria.
\newblock {\em Physical review letters}, 98(15):158102, 2007.

\bibitem{Topaz2008}
C.~M. Topaz, A.~J. Bernoff, S.~Logan, and W.~Toolson.
\newblock A model for rolling swarms of locusts.
\newblock {\em The European Physical Journal Special Topics}, 157(1):93--109,
  2008.

\bibitem{Topaz2006}
C.~M. Topaz, A.~L. Bertozzi, and M.~A. Lewis.
\newblock A nonlocal continuum model for biological aggregation.
\newblock {\em Bulletin of mathematical biology}, 68(7):1601, 2006.

\bibitem{Toscani2000}
G.~Toscani.
\newblock One-dimensional kinetic models of granular flows.
\newblock {\em ESAIM: Mathematical Modelling and Numerical Analysis},
  34(6):1277--1291, 2000.

\bibitem{Vicsek1995}
T.~Vicsek, A.~Czir{\'o}k, E.~Ben-Jacob, I.~Cohen, and O.~Shochet.
\newblock Novel type of phase transition in a system of self-driven particles.
\newblock {\em Phys. Rev. Lett.}, 75(6):1226--1229, 1995.

\bibitem{wrobel2016enhanced}
J.~K. Wr{\'o}bel, S.~Lynch, A.~Barrett, L.~Fauci, and R.~Cortez.
\newblock Enhanced flagellar swimming through a compliant viscoelastic network
  in stokes flow.
\newblock {\em Journal of Fluid Mechanics}, 792:775--797, 2016.

\end{thebibliography}

\end{document}